\documentclass[11pt,letterpaper]{amsart}

\usepackage{amsmath,amssymb,amsthm,mathtools}
\usepackage{aliascnt}
\usepackage{enumitem}
\usepackage{microtype}
\usepackage{xcolor}
\usepackage[hidelinks]{hyperref}
\usepackage[capitalise,nameinlink]{cleveref}
\usepackage{parskip}
\hypersetup{
  pdftitle={Sharp refined-direction Kakeya estimates in finite Heisenberg groups},
  pdfauthor={Thang Pham, Andrea Pinamonti, Dung The Tran, and Boqing Xue},
  pdfsubject={Finite field Kakeya estimates in Heisenberg groups}
}

\allowdisplaybreaks
\numberwithin{equation}{section}

\newtheorem{theorem}{Theorem}[section]
\newaliascnt{proposition}{theorem}
\newtheorem{proposition}[proposition]{Proposition}
\aliascntresetthe{proposition}
\newaliascnt{lemma}{theorem}
\newtheorem{lemma}[lemma]{Lemma}
\aliascntresetthe{lemma}
\newaliascnt{corollary}{theorem}
\newtheorem{corollary}[corollary]{Corollary}
\aliascntresetthe{corollary}

\theoremstyle{definition}
\newaliascnt{definition}{theorem}
\newtheorem{definition}[definition]{Definition}
\aliascntresetthe{definition}

\theoremstyle{remark}
\newaliascnt{remark}{theorem}
\newtheorem{remark}[remark]{Remark}
\aliascntresetthe{remark}

\crefname{proposition}{Proposition}{propositions}
\crefname{lemma}{Lemma}{lemmas}
\crefname{corollary}{Corollary}{corollaries}
\crefname{definition}{Definition}{definitions}
\crefname{remark}{Remark}{remarks}

\newcommand{\Fq}{\mathbb F_q}
\newcommand{\C}{\mathbb C}
\newcommand{\PP}{\mathbb P}
\newcommand{\HH}{\mathbb H}
\newcommand{\Dn}{\mathcal D_n}
\newcommand{\Mrd}{\mathcal M^{\mathrm{rd}}_{\HH_n}}
\newcommand{\one}{\mathbf 1}
\newcommand{\norm}[1]{\lVert #1\rVert}
\newcommand{\Dir}{\operatorname{Dir}}
\newcommand{\Sp}{\operatorname{Sp}}
\newcommand{\conv}{\operatorname{conv}}

\ExplSyntaxOn
\box_new:N \l__revision_word_box
\cs_new_protected:Npn \__revision_strike_word:n #1
 {
  \hbox_set:Nn \l__revision_word_box {#1}
  \leavevmode
  \rlap{\raisebox{.45ex}{\rule{\box_wd:N \l__revision_word_box}{.4pt}}}
  \box_use:N \l__revision_word_box
 }
\cs_new_protected:Npn \__revision_deleted:n #1
 {
  \group_begin:
  \color{red}
  \seq_set_split:Nnn \l_tmpa_seq { ~ } {#1}
  \seq_map_indexed_inline:Nn \l_tmpa_seq
   {
    \int_compare:nNnF {##1} = {1}
     {
      \penalty0
      \hskip\fontdimen2\font plus\fontdimen3\font
       minus\fontdimen4
     }
    \__revision_strike_word:n {##2}
   }
  \group_end:
 }

\ExplSyntaxOff

\title[Refined-direction Kakeya estimates]
{Sharp refined-direction Kakeya estimates\\
in finite Heisenberg groups}

\author{Thang Pham}
\address{Institute of Mathematics and Interdisciplinary Sciences,
Xidian University, China}
\email{thangpham.math@gmail.com}

\author{Andrea Pinamonti}
\address{Department of Mathematics, University of Trento, Italy}
\email{andrea.pinamonti@unitn.it}

\author{Dung The Tran}
\address{VNU University of Science, Hanoi, Vietnam}
\email{tranthedung56@gmail.com}

\author{Boqing Xue}
\address{Institute of Mathematical Sciences, ShanghaiTech University,
China}
\email{xuebq@shanghaitech.edu.cn}

\subjclass[2020]{Primary 42B25; Secondary 05B25, 51E20}
\keywords{finite Heisenberg group, Kakeya maximal operator,
refined direction, polynomial method, method of multiplicities,
Furstenberg set}

\begin{document}

\begin{abstract}
Let \(n\geq 2\) and let \(q\) be an odd prime power. The first aim of this
paper is to prove that, for every \(E\subset \mathbb{H}_n(\mathbb{F}_q)\)
and every \(\lambda>0\), the following sharp rich-direction estimate holds
\[
 \left|
 \left\{
 \vartheta\in\mathcal{D}_n:
 \mathcal{M}^{\mathrm{rd}}\mathbf{1}_E(\vartheta)\geq\lambda
 \right\}
 \right|
 \lesssim_n
 q^{2n-1}|E|\lambda^{-2n}.
\]
The second aim is to determine, for every \(1\leq u,v\leq\infty\), the
sharp exponent of \(q\) in the corresponding \(\ell^u\to\ell^v\)
estimate. More precisely, we prove that
\[
 A_n^{\mathrm{rd}}(u,v)
 =
 \max\left\{
 \frac{2n-1}{v},\
 1-\frac1u,\
 \frac{2n}{v}-\frac1u,\
 1+\frac{2n}{v}-\frac{2n+1}{u}
 \right\}.
\]
The proof combines the polynomial method with multiplicities and a
probabilistic covering argument based on the action of the affine
symplectic group.
\end{abstract}
\maketitle
\section{Introduction}
\label{sec:introduction}
Kakeya problems measure how efficiently lines pointing in different
directions can overlap. Over a finite field, a Kakeya set in $\Fq^d$ is a set containing an affine line in every direction. The connection between finite field Kakeya problems and harmonic analysis was developed by Mockenhaupt and Tao~\cite{MockenhauptTao04}, while Dvir's polynomial argument~\cite{Dvir09} showed that every finite field Kakeya set has cardinality comparable to that of the ambient space. The polynomial method, its multiplicity refinements, and the corresponding maximal estimates were subsequently developed in \cite{DKSS13,EOT10}. Related finite field Furstenberg problems, in which one prescribes only a portion of a line in a selected family of directions, were studied in \cite{DharDvirLund21,DharDvirLund24,EllenbergErman16}.

The present paper studies an analogous problem for horizontal lines in finite Heisenberg groups. This question has both a finite field and a sub-Riemannian motivation. In the real first Heisenberg group, Liu~\cite{Liu22} proved the sharp lower bound for the Hausdorff dimension of Heisenberg Kakeya sets, and Venieri~\cite{Venieri14} related maximal
estimates to dimension bounds for Besicovitch sets.  More recently, F\"assler, Pinamonti, and Wald~\cite{FasslerPinamontiWald25} proved a Kakeya maximal inequality for horizontal tubes and recovered Liu's
dimension bound.

The finite field problem considered here originates in our previous
work~\cite{PPTX26}. It is useful to distinguish the two maximal
operators studied there. For the operator parameterized only by
projective horizontal directions, projection onto the horizontal layer $\Fq^{2n}$ reduces the problem to the ordinary affine finite field Kakeya maximal operator; this gives the exact $\ell^u\to\ell^v$ growth exponents for every $n$. The genuinely Heisenberg problem appears when one also records the central slope of a horizontal line. This leads to the \emph{refined-direction} operator. In~\cite{PPTX26}, the refined problem was solved sharply for $\HH_1(\Fq)$. The proof uses the planar Kakeya estimate, Plancherel's theorem, character orthogonality, and a bounded-fiber estimate for an explicit quadratic map, and does not use polynomial vanishing. It is worth noting that the Fourier-analytic framework in \cite{PPTX26} was developed with the broader hope of finding a purely non-algebraic proof for the affine Kakeya problem in $\mathbb{F}_q^3$.

A central point of the present paper is that the refined-direction problem for $n\geq2$ is not a formal extension of the $n=1$ argument.  
When $n=1$, the critical diagonal is $\ell^2\to\ell^2$, exactly the regime in which Plancherel and character orthogonality are available. For $n\geq2$, the critical diagonal becomes
$
 \ell^{2n}\longrightarrow\ell^{2n}.
$
As observed in~\cite[Section~1]{PPTX26}, a direct extension of the $n=1$ Fourier argument to $n\geq2$ yields only an $\ell^2\to\ell^{2n}$ estimate with factor $q^{n/2}$.  By contrast, the point-mass example shows that the exponent $(2n-1)/(2n)$ is necessary, and the level-set theorem in this paper attains this power exactly. Thus, for $n\geq2$, the difficulty is not merely a loss in the power of $q$ in the $n=1$ Fourier argument. Rather, the critical estimate itself moves from the $\ell^2$ setting to the genuinely higher-moment $\ell^{2n}$ setting, requiring a different mechanism.

\subsection{Horizontal lines and refined directions}

Fix $n\geq2$ and set
\[
 r:=2n,
 \qquad
 V:=\Fq^r=\Fq^n\times\Fq^n.
\]
We equip $V$ with the standard nondegenerate alternating form
\begin{equation}\label{eq:symplectic-form}
 \sigma((x,y),(x',y')):=x\cdot y'-y\cdot x'.
\end{equation}
Throughout, $q$ is an odd prime power and $n$ is fixed.  We realize the
finite Heisenberg group as
\[
 \HH_n(\Fq)=V\times\Fq
\]
with group law
\begin{equation}\label{eq:group-law}
 (z,t)\cdot(z',t')
 =
 (z+z',t+t'+\sigma(z,z')).
\end{equation}
If $p=(z_0,t_0)\in\HH_n(\Fq)$ and $w\in V\setminus\{0\}$, the horizontal
line through $p$ with spatial direction $[w]$ is
\begin{equation}\label{eq:horizontal-line}
 L_{p,w}
 =
 \bigl\{
   (z_0+sw,t_0+s\sigma(z_0,w)):s\in\Fq
 \bigr\}.
\end{equation}
The spatial direction $[w]\in\PP^{r-1}(\Fq)$ does not record the central
slope of the line. Following \cite{PPTX26}, we therefore define the
\emph{refined direction} by
\begin{equation}\label{eq:refined-direction}
 \Dir(L_{p,w})
 =
 [w:\sigma(z_0,w)]
 \in\PP^r(\Fq).
\end{equation}
This projective point is independent of the chosen base point on the
line and of the representative of $[w]$.  Moreover, since
$L_{p,w}$ is an ordinary affine line in $V\times\Fq$ with direction
vector $(w,\sigma(z_0,w))$, its refined direction is exactly its ambient
projective direction.

The space of refined directions is
\begin{equation}\label{eq:direction-space}
 \Dn
 :=
 \bigl\{
 [w:c]\in\PP^r(\Fq):w\neq0
 \bigr\}.
\end{equation}
Every element of $\Dn$ occurs as the refined direction of a horizontal
line, because $z\mapsto\sigma(z,w)$ is surjective whenever $w\neq0$.
We shall use the cardinalities
\begin{equation}\label{eq:cardinalities}
 \begin{aligned}
 |\HH_n(\Fq)|&=q^{r+1},\\
 |\PP^{r-1}(\Fq)|&=1+q+\cdots+q^{r-1},\\
 |\Dn|&=q+q^2+\cdots+q^r.
 \end{aligned}
\end{equation}
For $F:\HH_n(\Fq)\to\C$, define the refined-direction maximal operator
by
\begin{equation}\label{eq:maximal-operator}
 (\Mrd F)(\vartheta)
 :=
 \max_{\substack{L\ \mathrm{horizontal}\\
                 \Dir(L)=\vartheta}}
 \sum_{p\in L}|F(p)|,
 \qquad
 \vartheta\in\Dn.
\end{equation}
All $\ell^p$-norms below are taken with respect to the counting measure.
We write $X\lesssim Y$ if $X\leq CY$ for an absolute constant $C$, and
$X\lesssim_nY$ if the constant may depend on $n$ but not on $q$ or on
the sets, functions, and parameters under consideration.  The notation
$X\approx Y$ has the corresponding two-sided meaning.

For $E\subset\HH_n(\Fq)$ and $\lambda>0$, set
\[
 \Omega_\lambda(E)
 :=
 \bigl\{
   \vartheta\in\Dn:
   \Mrd\one_E(\vartheta)\geq\lambda
 \bigr\}.
\]
Thus, $\Omega_\lambda(E)$ is the set of refined directions for which
some horizontal line contains at least $\lambda$ points of $E$.

\subsection{Sharp estimates for rich refined directions}

Our principal geometric estimate gives uniform control of rich refined directions.
\begin{theorem}
\label{thm:rich-direction-estimate}
For every $E\subset\HH_n(\Fq)$ and every $\lambda>0$,
\begin{equation}\label{eq:rich-direction-bound}
 |\Omega_\lambda(E)|
 \lesssim_n
 q^{r-1}|E|\lambda^{-r}.
\end{equation}
The exponent $r-1$ of $q$ in
\eqref{eq:rich-direction-bound} is sharp.
\end{theorem}
The point-mass example already forces the power in
\eqref{eq:rich-direction-bound}: every point of $\HH_n(\Fq)$ lies on
horizontal lines in
$|\PP^{r-1}(\Fq)|\approx q^{r-1}$ distinct refined directions.
A significant feature of \cref{thm:rich-direction-estimate} is that
$\Omega_\lambda(E)$ is completely arbitrary; no algebraic, density, or
separation assumption is imposed on the set of rich directions.

The same mechanism has a weighted form and yields the corresponding sharp level-set estimate for arbitrary functions.

\begin{theorem}
\label{thm:function-level-set}
For every $F:\HH_n(\Fq)\to\C$ and every $\lambda>0$,
\begin{equation}\label{eq:function-level-set}
 \left|
 \left\{
   \vartheta\in\Dn:
   \Mrd F(\vartheta)\geq\lambda
 \right\}
 \right|
 \lesssim_n
 q^{r-1}\lambda^{-r}
 \norm{F}_{\ell^r(\HH_n(\Fq))}^r.
\end{equation}
The exponent $r-1$ of $q$ in
\eqref{eq:function-level-set} is sharp.
\end{theorem}

\subsection{Geometric consequences}

\Cref{thm:rich-direction-estimate}
immediately gives the following finite field Furstenberg type
estimate.
\begin{theorem}
\label{thm:furstenberg-application}
Let $\varnothing\neq\Omega\subset\Dn$, let
$E\subset\HH_n(\Fq)$, and let $1\leq\lambda\leq q$ be an integer.
Suppose that for every $\vartheta\in\Omega$ there exists a horizontal
line $L_\vartheta$ satisfying
\[
 \Dir(L_\vartheta)=\vartheta,
 \qquad
 |E\cap L_\vartheta|\geq\lambda.
\]
Then
\begin{equation}\label{eq:furstenberg-endpoint}
 |E|
 \gtrsim_n
 \frac{\lambda^r|\Omega|}{q^{r-1}}.
\end{equation}
\end{theorem}
The strength of \cref{thm:furstenberg-application} lies in its
uniformity in both $\Omega$ and $\lambda$.  For example, if
$|\Omega|\geq\delta|\Dn|$ and each selected line contains at least
$\eta q$ points of $E$, then
\begin{equation}\label{eq:robust-density}
 |E|
 \gtrsim_n
 \delta\,\eta^r q^{r+1}.
\end{equation}
Thus, a positive proportion of rich refined directions forces a
quantitatively positive ambient density. Theorem \ref{thm:furstenberg-application} is most interesting when $1<\lambda<q$. When $\lambda=q$ in \eqref{eq:furstenberg-endpoint}, it gives the particularly simple refined-direction Kakeya estimate
\begin{equation}\label{eq:partial-direction-kakeya}
 |E|\gtrsim_n q|\Omega|
\end{equation}
whenever $E$ contains a complete horizontal line in every direction
$\vartheta\in\Omega$.  

\begin{corollary}
    If $\mathcal L$ is a family of
horizontal lines with pairwise distinct refined directions, then
\begin{equation}\label{eq:union-distinct-directions}
 \left|\bigcup_{L\in\mathcal L}L\right|
 \approx_n q|\mathcal L|.
\end{equation}
\end{corollary}
The upper bound of $\left|\bigcup_{L\in\mathcal L}L\right|$ is trivial, whereas the lower bound is
\eqref{eq:partial-direction-kakeya}. The lower bound guaranties that no matter how these lines are arranged, as long as they point in pairwise distinct refined directions, they cannot overlap so heavily that their union collapses in size.
\subsection{The critical exponent diagram}

We next determine the optimal power of $q$ throughout the full
mixed-norm range and prove that this pure power is attained
away from a single critical diagonal point.

\begin{definition}\label{def:critical-exponent}
For $1\leq u,v\leq\infty$, let $A_n^{\mathrm{rd}}(u,v)$ be the infimum
of all $a\geq0$ for which there exists a constant
$C=C(n,u,v,a)$ such that
\begin{equation}\label{eq:def-critical-exponent}
 \norm{\Mrd F}_{\ell^v(\Dn)}
 \leq
 Cq^a
 \norm{F}_{\ell^u(\HH_n(\Fq))}
\end{equation}
for every odd prime power $q$ and every
$F:\HH_n(\Fq)\to\C$.
\end{definition}

The use of the infimum in this definition is deliberate. The theorem below shows that it is attained for every pair $(u,v)\neq(r,r)$.  At $(r,r)$, the exact critical power is known, but whether the corresponding pure-power estimate, i.e., without the log factor, holds remains open.  This is the precise distinction from the case $n=1$, where the Fourier estimate of
\begin{equation}\label{eq:intro-n1-estimate}
 \norm{\mathcal M^{\mathrm{rd}}_{\HH_1}F}_{\ell^2(\mathcal D_1)}
 \lesssim
 q^{1/2}\norm{F}_{\ell^2(\HH_1(\Fq))}
\end{equation} makes the critical value a minimum throughout the mixed-norm range.

Write
\[
 \alpha=\frac1u,
 \qquad
 \beta=\frac1v,
\]
with the conventions $1/\infty=0$ and $1/0=\infty$, and set
\[
 \widetilde A_n^{\mathrm{rd}}(\alpha,\beta)
 :=
 A_n^{\mathrm{rd}}(1/\alpha,1/\beta).
\]
Define
\begin{equation}\label{eq:four-functionals}
 \begin{aligned}
 \phi_1(\alpha,\beta)&=(r-1)\beta,
 &
 \phi_2(\alpha,\beta)&=1-\alpha,\\
 \phi_3(\alpha,\beta)&=r\beta-\alpha,
 &
 \phi_4(\alpha,\beta)&=1+r\beta-(r+1)\alpha,
 \end{aligned}
\end{equation}
and
\begin{equation}\label{eq:Phi}
 \Phi_n(\alpha,\beta)
 :=
 \max\{
   \phi_1(\alpha,\beta),
   \phi_2(\alpha,\beta),
   \phi_3(\alpha,\beta),
   \phi_4(\alpha,\beta)
 \}.
\end{equation}

\begin{theorem}
\label{thm:main}
For every $(\alpha,\beta)\in[0,1]^2$,
\begin{equation}\label{eq:main-phase}
 \widetilde A_n^{\mathrm{rd}}(\alpha,\beta)
 =
 \Phi_n(\alpha,\beta).
\end{equation}
Equivalently, for every $1\leq u,v\leq\infty$,
\begin{equation}\label{eq:main-uv}
 A_n^{\mathrm{rd}}(u,v)
 =
 \max\left\{
 \frac{2n-1}{v},\
 1-\frac1u,\
 \frac{2n}{v}-\frac1u,\
 1+\frac{2n}{v}-\frac{2n+1}{u}
 \right\}.
\end{equation}
More precisely, for every fixed
$(\alpha,\beta)\in[0,1]^2\setminus\{(1/r,1/r)\}$,
\begin{equation}\label{eq:main-quantitative}
 \norm{\Mrd F}_{\ell^{1/\beta}(\Dn)}
 \lesssim_{n,\alpha,\beta}
 q^{\Phi_n(\alpha,\beta)}
 \norm{F}_{\ell^{1/\alpha}(\HH_n(\Fq))}.
\end{equation}
At the critical diagonal,
\begin{equation}\label{eq:main-quantitative-diagonal}
 \norm{\Mrd F}_{\ell^r(\Dn)}
 \lesssim_n
 q^{(r-1)/r}(1+\log q)^{1/r}
 \norm{F}_{\ell^r(\HH_n(\Fq))}.
\end{equation}
No estimate with a smaller power of $q$ can hold uniformly in $q$.
Consequently, the infimum defining $A_n^{\mathrm{rd}}(u,v)$ is attained for every $(u,v)\neq(r,r)$.
\end{theorem}
The four affine functionals in \eqref{eq:main-uv} arise from four
different geometric obstructions.  They are forced, respectively, by a point mass, the indicator of one horizontal line, a set of $O_n(q)$ points in the central slice $\{t=0\}$ whose projection meets every affine hyperplane in $V$, and the constant function.  These examples give the matching lower bounds throughout the exponent diagram.

\subsection{Comparison with the result for the first Heisenberg group}

We now make the relation with~\cite{PPTX26} more precise.

As sets, $\HH_1(\Fq)=\Fq^3$, and the refined directions are precisely
the nonvertical ambient projective directions.  The admissible lines,
however, form only a proper subfamily of all affine lines.  A line
through $(x_0,y_0,t_0)$ in direction
$(\xi_1,\xi_2,\xi_3)$ is horizontal exactly when
\[
 \xi_3=x_0\xi_2-y_0\xi_1.
\]
Consequently, even when $n=1$, the refined-direction estimate is not the ordinary three-dimensional finite field Kakeya estimate.

The methodological source of the distinction is clear. When $n=1$, the critical diagonal is $\ell^2\to\ell^2$, and the Fourier argument of~\cite{PPTX26} gives the optimal strong estimate directly.  For $n\geq2$, the critical diagonal becomes $\ell^{2n}\to\ell^{2n}$.  Every point lies on $\approx q^{2n-1}$ horizontal lines with distinct refined directions and therefore forces the power $q^{(2n-1)/(2n)}$.  A direct continuation of the $n=1$ argument remains $\ell^2$-based and yields only an $\ell^2\to\ell^{2n}$ estimate with the larger factor $q^{n/2}$.

The present paper supplies the genuinely higher-moment mechanism needed in this setting. Horizontal polynomial multiplicity and the global covering argument give the sharp level-set estimate at the critical $\ell^{2n}$ scaling. Rearrangement, real interpolation, and Riesz--Thorin interpolation of the selector operators then yield the optimal pure power of $q$ at every mixed-norm pair away from the critical diagonal. Only the estimate at the critical diagonal itself retains the logarithmic factor.

Thus, the passage from $\HH_1(\Fq)$ to $\HH_n(\Fq)$ is not merely a
dimensional generalization. It marks a transition from a problem resolved by Fourier analysis to one requiring a new higher-moment polynomial method.

\subsection{Outline of the proof}

We conclude the introduction by indicating the mechanisms behind
\cref{thm:rich-direction-estimate}.

The polynomial part of the argument is local in the refined-direction
space. Consider the affine chart
\[
 \mathcal D_n^{\mathrm{aff}}
 =
 \{[w:1]:w\in V\setminus\{0\}\}.
\]
A horizontal line with refined direction $[w:1]$ has a unique
representation
\[
 L_w(\rho)
 =
 \{(\rho+sw,s):s\in\Fq\},
 \qquad
 \sigma(\rho,w)=1;
\]
see \cref{lem:normalized-horizontal-lines}. At a point
$p=(z,t)\in\HH_n(\Fq)$, the space of horizontal direction vectors is
\[
 \mathcal H_p
 =
 \{(\xi,\sigma(z,\xi)):\xi\in V\},
\]
which has dimension $r$.

We impose multiplicity only along $\mathcal H_p$. More precisely,
\cref{def:horizontal-multiplicity} requires the Hasse--Taylor
components of degree less than $m$ to vanish after restriction to
$\mathcal H_p$.  By \cref{lem:multiplicity-condition-count}, this
imposes only
\[
 \binom{m+r-1}{r}
 =
 O_n(m^r)
\]
linear conditions at each point, instead of the $O_n(m^{r+1})$ conditions required by ordinary ambient multiplicity.
By comparing this condition count to the dimension of the polynomial space, the interpolation lemma (\cref{lem:horizontal-interpolation}) guarantees the existence of a nonzero polynomial $P$ satisfying these constraints, with a bounded degree
\[
 \deg P
 \lesssim_n
 |E|^{1/(r+1)}m^{r/(r+1)}.
\]

If $L_w(\rho)$ contains at least $\lambda$ points of $E$, then the
restriction
\[
 s\longmapsto P(\rho+sw,s)
\]
has at least $\lambda$ distinct zeros, each of multiplicity at least
$m$ (\cref{lem:line-multiplicity}). Choosing $m$ so that
$m\lambda>\deg P$ forces this restriction to vanish identically. Its
leading coefficient is the value at $(w,1)$ of the top homogeneous
part of $P$. The Schwartz-Zippel lemma then gives
\begin{equation}\label{eq:intro-polynomial-estimate}
 |\widetilde\Omega_\lambda(E)|
 \lesssim_n
 q^{r-1}
 \max\left\{
 |E|^{1/(r+1)},
 \frac{|E|}{\lambda^r}
 \right\},
\end{equation}
which is \cref{prop:polynomial-rich-line}.  In particular,
\cref{cor:positive-proportion-directions} yields
\begin{equation}\label{eq:intro-positive-proportion}
 |\widetilde\Omega_\lambda(E)|
 \gtrsim q^r
 \qquad\Longrightarrow\qquad
 |E|\gtrsim_n q\lambda^r.
\end{equation}

The second part of the argument globalizes
\eqref{eq:intro-positive-proportion}.  The affine symplectic
transformations
\[
 g_{b,\tau,S}(z,t)
 :=
 \bigl(
 b+Sz,\,
 \tau+t+\sigma(b,Sz)
 \bigr),
 \qquad
 b\in V,\quad
 \tau\in\Fq,\quad
 S\in\Sp(V,\sigma),
\]
preserve horizontal lines and act transitively on $\Dn$ (\cref{lem:symmetry-action}). Given a nonempty set of directions
$\Omega\subset\Dn$, the probabilistic covering lemma
\cref{lem:covering-direction-set} produces $O(q^r/|\Omega|)$ such transformations whose images of $\Omega$ cover
a positive proportion of the affine chart.

Apply this with $\Omega=\Omega_\lambda(E)$.  If
$\Omega_\lambda(E)=\varnothing$, then \eqref{eq:rich-direction-bound} is
immediate.  Otherwise, choose $g_1,\ldots,g_J$ as in
\cref{lem:covering-direction-set}, where
\[
 J\lesssim_n\frac{q^r}{|\Omega_\lambda(E)|},
 \qquad
 E_{\mathrm{cov}}:=\bigcup_{j=1}^Jg_j(E).
\]
Covariance gives
\[
 |\widetilde\Omega_\lambda(E_{\mathrm{cov}})|
 \geq \frac12q^r.
\]
Since each $g_j$ is a bijection, the union bound also gives
\[
 |E_{\mathrm{cov}}|
 \leq J|E|
 \lesssim_n
 \frac{q^r|E|}{|\Omega_\lambda(E)|}.
\]
It follows from \eqref{eq:intro-positive-proportion} that
\[
 q\lambda^r
 \lesssim_n
 |E_{\mathrm{cov}}|
 \lesssim_n
 \frac{q^r|E|}{|\Omega_\lambda(E)|}.
\]
Rearranging gives
\[
 |\Omega_\lambda(E)|
 \lesssim_n
 q^{r-1}|E|\lambda^{-r},
\]
which is precisely \eqref{eq:rich-direction-bound}.
A weighted version of the polynomial construction,
\cref{prop:weighted-affine-rich}, together with the same covering
argument gives the integer-valued level-set estimate
\cref{prop:integer-weighted-level} and hence
\cref{thm:function-level-set}. Ordering the values of the
maximal function shows that only the diagonal $\ell^r$ sum is harmonic:
it gives \cref{thm:diagonal-estimate} with the factor
$(1+\log q)^{1/r}$ and gives pure-power estimates at every other point on the critical-input line. Marcinkiewicz interpolation from the
level-set estimate supplies strong estimates on the open edges issuing from the critical point. Riesz--Thorin interpolation for the selector operators then extends these bounds across the four convex cells that partition the mixed-norm parameter space $[0, 1]^2$. Together with the four sharpness examples in \cref{sec:lower-bounds}, this proves \cref{thm:main}.

\subsection{Organization of the paper}

\Cref{sec:preliminaries} contains the elementary geometry of horizontal lines, the selector formulation, and the complex and real
interpolation principles used below.
The four examples giving the necessary lower bounds for the critical exponent diagram are collected in \cref{sec:lower-bounds}.
The horizontal polynomial method is developed in \cref{sec:polynomial-estimate}, and the affine-symplectic covering
argument in \cref{sec:covering-directions}. These two ingredients are combined in \cref{sec:global-rich-direction} to prove
\cref{thm:rich-direction-estimate}, and their weighted counterparts are developed in \cref{sec:function-level-set} to prove
\cref{thm:function-level-set}. The diagonal $\ell^r$ estimate and the complete critical exponent diagram are proved in
\cref{sec:diagonal,sec:mixed-norm-estimates}.

\section{Preliminaries}
\label{sec:preliminaries}
This section develops the preliminary tools used throughout
the paper.
\begin{lemma}\label{lem:line-geometry}
Every point of $\HH_n(\Fq)$ lies on exactly $|\PP^{r-1}(\Fq)|$ horizontal lines, and these lines have pairwise distinct refined directions. Two distinct points lie on at most one common horizontal line.
\end{lemma}

\begin{proof}
For every spatial direction $[w]\in\PP^{r-1}(\Fq)$, there is exactly one horizontal line through a fixed point $p$ with spatial direction $[w]$. Its refined direction is determined by $p$ and $[w]$, and two such refined directions can agree only if the spatial directions agree.  This proves the first assertion.

Let $p=(z,t)$ and $p'=(z',t')$ be distinct points on a common
horizontal line. Since the spatial parametrization in
\eqref{eq:horizontal-line} is injective, one has $z'\neq z$, and the
spatial direction of any common line must be
\[
 [w]=[z'-z].
\]
By the first part, $p$ and this spatial direction determine a unique
horizontal line. More explicitly, such a line exists precisely when
\[
 t'-t=\sigma(z,z'-z),
\]
in which case it is $L_{p,z'-z}$. Hence two distinct points lie on at most one common horizontal line. This proves the second assertion.

\end{proof}

\begin{definition}
A selector is a family
\[
 \mathcal S=(L_\vartheta)_{\vartheta\in\Dn},
 \qquad
 \Dir(L_\vartheta)=\vartheta.
\]
Its associated linear operator is
\begin{equation}\label{eq:selector}
 (T_{\mathcal S}f)(\vartheta)
 :=
 \sum_{p\in L_\vartheta}f(p).
\end{equation}
\end{definition}

For every selector,
\begin{equation}\label{eq:selector-dominated}
 |T_{\mathcal S}f|\leq\Mrd(|f|).
\end{equation}
Conversely, for every $f\geq0$, one may choose a maximizing selector
$\mathcal S_f$ such that
\begin{equation}\label{eq:max-selector}
 \Mrd f=T_{\mathcal S_f}f.
\end{equation}
Thus, uniform estimates for all selector operators imply the corresponding maximal estimates.

We use the following standard interpolation principles. Suppose first that $T$ is linear and
\[
 \norm{Tf}_{\ell^{v_j}}
 \leq C_jq^{a_j}\norm{f}_{\ell^{u_j}},
 \qquad j=0,1.
\]
If
\[
 \frac1{u_\theta}=\frac{1-\theta}{u_0}+\frac\theta{u_1},
 \qquad
 \frac1{v_\theta}=\frac{1-\theta}{v_0}+\frac\theta{v_1},
\]
then Riesz--Thorin interpolation gives
\[
 \norm{Tf}_{\ell^{v_\theta}}
 \leq C_0^{1-\theta}C_1^\theta
 q^{(1-\theta)a_0+\theta a_1}
 \norm{f}_{\ell^{u_\theta}}.
\]
We apply this principle to the linear operators $T_{\mathcal S}$.  All estimates below are uniform in $\mathcal S$, so a maximizing selector may be chosen after interpolation. Repeated two-point interpolation realizes any finite convex combination of exponent pairs.

We shall also use the following consequence of real interpolation.
Let $\mathcal X$ and $\mathcal Y$ be finite sets equipped with counting
measure, and let
\[
 T:\mathbb C^{\mathcal X}\longrightarrow\mathbb C^{\mathcal Y}
\]
be linear. Suppose that
\begin{equation}\label{eq:distributional-interpolation-endpoint}
 \bigl|\{y\in\mathcal Y:|Tf(y)|\geq\lambda\}\bigr|
 \leq
 \left(
   \frac{M_0\norm{f}_{\ell^{p_0}(\mathcal X)}}{\lambda}
 \right)^{q_0}
 \qquad(\lambda>0)
\end{equation}
and
\begin{equation}\label{eq:strong-interpolation-endpoint}
 \norm{Tf}_{\ell^{q_1}(\mathcal Y)}
 \leq
 M_1\norm{f}_{\ell^{p_1}(\mathcal X)},
\end{equation}
where $1\leq p_i,q_i\leq\infty$, $q_0<\infty$, and
$q_0\neq q_1$. For $0<\theta<1$, set
\[
 \frac1{p_\theta}
 =
 \frac{1-\theta}{p_0}+\frac\theta{p_1},
 \qquad
 \frac1{q_\theta}
 =
 \frac{1-\theta}{q_0}+\frac\theta{q_1}.
\]
Real interpolation, with secondary index $p_\theta$, gives
\[
 (\ell^{p_0},\ell^{p_1})_{\theta,p_\theta}
 =\ell^{p_\theta}
\]
and, since $q_0\neq q_1$,
\[
 (\ell^{q_0,\infty},\ell^{q_1})_{\theta,p_\theta}
 =\ell^{q_\theta,p_\theta},
\]
with equivalence of norms. Consequently, whenever
$p_\theta\leq q_\theta$, the Lorentz embedding
$\ell^{q_\theta,p_\theta}\hookrightarrow\ell^{q_\theta}$ yields
\begin{equation}\label{eq:marcinkiewicz-interpolation}
 \norm{Tf}_{\ell^{q_\theta}(\mathcal Y)}
 \lesssim_{p_0,p_1,q_0,q_1,\theta}
 M_0^{1-\theta}M_1^\theta
 \norm{f}_{\ell^{p_\theta}(\mathcal X)}.
\end{equation}
We use \eqref{eq:marcinkiewicz-interpolation} only with
$(p_0,q_0)=(r,r)$ and $(p_1,q_1)\in\{(\infty,\infty),(1,\infty),(1,1)\}$.

For a finite set $\mathcal X$ and a function
$\psi:\mathcal X\to\C$, we shall also use the norm comparison
\begin{equation}\label{eq:finite-norm-comparison}
 \norm{\psi}_{\ell^v(\mathcal X)}
 \leq\norm{\psi}_{\ell^u(\mathcal X)}
 \leq |\mathcal X|^{1/u-1/v}\norm{\psi}_{\ell^v(\mathcal X)},
 \qquad 1\leq u\leq v\leq\infty.
\end{equation}

\begin{lemma}
\label{lem:convexity}
The function $(\alpha,\beta)\mapsto
\widetilde A_n^{\mathrm{rd}}(\alpha,\beta)$ is
convex on $[0,1]^2$.
\end{lemma}

\begin{proof}
Let $\mathbf x_j=(\alpha_j,\beta_j)$ and put
$a_j=\widetilde A_n^{\mathrm{rd}}(\mathbf x_j)$ for $j=0,1$. Given
$\varepsilon>0$, the definition of the infimum supplies bounds with powers $q^{a_j+\varepsilon}$ at $\mathbf x_j$. By \eqref{eq:selector-dominated}, the same bounds hold uniformly for every linearized operator $T_{\mathcal S}$. Riesz--Thorin interpolation at $\mathbf x_\theta=(1-\theta)\mathbf x_0+\theta\mathbf x_1$ gives a selector bound with power
\[
 q^{(1-\theta)a_0+\theta a_1+\varepsilon}.
\]
Apply it to $|f|$, choose a maximizing selector, and let
$\varepsilon\downarrow0$.
This proves 
\[
 \widetilde A_n^{\mathrm{rd}}(\mathbf x_\theta)
 \leq(1-\theta)\widetilde A_n^{\mathrm{rd}}(\mathbf x_0)
       +\theta\widetilde A_n^{\mathrm{rd}}(\mathbf x_1).
\]
The corresponding assertion for an arbitrary finite convex combination follows by induction.
\end{proof}

\section{Necessary lower bounds}
\label{sec:lower-bounds}

\begin{proposition}\label{prop:lower-bounds}
For every $(\alpha,\beta)\in[0,1]^2$,
\begin{equation}\label{eq:lower-bounds}
 \widetilde A_n^{\mathrm{rd}}(\alpha,\beta)
 \geq\Phi_n(\alpha,\beta).
\end{equation}
\end{proposition}

\begin{proof}
Fix $u=\frac{1}{\alpha}$ and $v=\frac{1}{\beta}$, with the usual convention when $\alpha=0$ or $\beta=0$. If \eqref{eq:def-critical-exponent} holds with the power $q^a$, then every nonzero test function $F$ satisfies
\begin{equation}\label{eq:test-function-principle}
 q^a
 \gtrsim_{n,u,v,a}
 \frac{\norm{\Mrd F}_{\ell^v(\Dn)}}
      {\norm{F}_{\ell^u(\HH_n(\Fq))}}.
\end{equation}
We give four test functions and estimate the corresponding quotients from below.

\emph{A point mass.}
Fix $\mathbf{p}_0=(\mathbf{z}_0,t_0)\in\HH_n(\Fq)$ and set
$F=\one_{\{\mathbf{p}_0\}}$.  By \cref{lem:line-geometry},
$\mathbf{p}_0$ lies on exactly
\[
 |\PP^{r-1}(\Fq)|=1+q+\cdots+q^{r-1} \geq q^{r-1}
\]
horizontal lines with pairwise distinct refined directions. On each of these directions, $\Mrd F=1$. Since $\norm{F}_{\ell^u(\HH_n(\Fq))}=1$, it follows that
\[
 \frac{\norm{\Mrd F}_{\ell^v(\Dn)}}
      {\norm{F}_{\ell^u(\HH_n(\Fq))}}
 \geq |\PP^{r-1}(\Fq)|^{\frac{1}{v}}
 \geq q^{\frac{r-1}{v}}.
\]
Thus, every admissible power satisfies $a\geq\phi_1(\alpha,\beta)$.

\emph{One horizontal line.}
Let $L$ be a horizontal line and take $F=\one_L$. The standard
parametrization of $L$ is injective in the field parameter, so $|L|=q$
and
\[
 \norm{F}_{\ell^u(\HH_n(\Fq))}=q^{\frac{1}{u}}.
\]
At $\vartheta_0=\Dir(L)$, the line $L$ itself is admissible in the
definition of the maximal operator. Thus,
\[
 \Mrd F(\vartheta_0)
 \geq\sum_{\mathbf{p}\in L}\one_L(\mathbf{p})
 =q,
\]
and every $\ell^v$-norm dominates this single coordinate. Hence,
\[
 \frac{\norm{\Mrd F}_{\ell^v(\Dn)}}
      {\norm{F}_{\ell^u(\HH_n(\Fq))}}
 \geq q^{1-\frac{1}{u}}.
\]
Hence, $a\geq\phi_2(\alpha,\beta)$.

\emph{A small set realizing every refined direction.}
We seek a set of $O_n(q)$ points in the central slice $V\times\{0\}$ that meets a horizontal line in every refined direction. A point $(z,0)$ lies on a horizontal line of refined direction $[w:c]$ precisely when $\sigma(z, w)=c$. Thus, it is enough to choose a small subset of $V$ that meets every affine hyperplane arising from such an equation.

Fix a basis $e_1,\ldots,e_r$ of $V$ and define
\begin{equation}\label{eq:axes-set}
 \mathcal B
 :=
 \bigcup_{j=1}^r
 \{s e_j:s\in\Fq\}.
\end{equation}
Then $|\mathcal B|=1+r(q-1) \leq r q$, and $\mathcal B$ meets every affine hyperplane in $V$. Indeed, if $\varphi:V\to\Fq$ is a nonzero linear functional and $\Pi=\{z:\varphi(z)=c\}$, choose $j$ with
$\varphi(e_j)\neq 0$; then
$\frac{c}{\varphi(e_j)} \cdot e_j\in\Pi\cap\mathcal B$.

Let $E:=\{(z,0): z\in\mathcal B\}$ and $F:=\one_E$. 
Fix an arbitrary $\vartheta=[w:c]\in\Dn$. Since
$w\neq0$ and $\sigma$ is nondegenerate, the functional
\[
 z\longmapsto\sigma(z,w)
\]
is nonzero. Consequently,
\[
 \Pi_\vartheta
 :=\{z\in V:\sigma(z, w)=c\}
\]
is an affine hyperplane. This hyperplane is independent of the chosen representative of the projective point $[w:c]$. Choose $z_\vartheta\in\mathcal B\cap\Pi_\vartheta$. Then
\begin{equation}\label{eq:witnessing-line-final}
 L_{(z_\vartheta,0), w}
 =
 \{(z_\vartheta+sw,sc):s\in\Fq\}
\end{equation}
has refined direction $
 [w:\sigma(z_\vartheta,w)]
 =[w:c]
 =\vartheta$, 
and it contains $(z_\vartheta,0)\in E$. Consequently,
\[
 \Mrd F(\vartheta)\geq1
 \qquad(\vartheta\in\Dn),
\]
which together with $|E|=|\mathcal{B}| \leq r q$ implies that
\[
 \frac{\norm{\Mrd F}_{\ell^v(\Dn)}}
      {\norm{F}_{\ell^u(\HH_n(\Fq))}}
      \geq
 \frac{|\Dn|^{\frac{1}{v}}}{|E|^{\frac{1}{u}}}
 \geq
 r^{-\frac{1}{u}}q^{\frac{r}{v}-\frac{1}{u}}
 \gtrsim_n q^{\frac{r}{v}-\frac{1}{u}}.
\]
Thus, $a\geq\phi_3(\alpha,\beta)$.

\emph{The constant function.}
Finally, take $F\equiv1$ on $\HH_n(\Fq)$. Since
$|\HH_n(\Fq)|=q^{r+1}$,
\[
 \norm{F}_{\ell^u(\HH_n(\Fq))}
 =q^{\frac{r+1}{u}}.
\]
Every refined direction is realized by a horizontal line, and every
horizontal line contains $q$ points. Hence,
\[
 \Mrd F(\vartheta)=q
 \qquad(\vartheta\in\Dn).
\]
Therefore, by the third identity in \eqref{eq:cardinalities},
\[
 \frac{\norm{\Mrd F}_{\ell^v(\Dn)}}
      {\norm{F}_{\ell^u(\HH_n(\Fq))}}
 =
 \frac{q|\Dn|^{\frac{1}{v}}}{q^{\frac{r+1}{u}}}
 \geq
 q^{1+\frac{r}{v}-\frac{r+1}{u}}
 =
 q^{\phi_4(\alpha,\beta)}.
\]
Hence, $a\geq\phi_4(\alpha,\beta)$. Every admissible $a$ therefore
dominates all four affine functions, so
\[
 a\geq\max_{1\leq j\leq4}\phi_j(\alpha,\beta)
 =\Phi_n(\alpha,\beta).
\]
Taking the infimum over admissible powers proves the proposition.
\end{proof}

\section{Polynomial estimates for rich horizontal directions}
\label{sec:polynomial-estimate}

The purpose of this section is to prove the estimate needed when many rich directions lie in one affine chart of the refined-direction space. We first describe the horizontal lines in this chart. We then define multiplicity only in horizontal directions, construct a polynomial with the required multiplicity at every point of a given set, and use its leading homogeneous part to count the rich directions.

\subsection{Normalized lines in the affine direction chart}

Let
\begin{equation}\label{eq:affine-chart-directions}
 \Dn^{\mathrm{aff}}
 :=
 \bigl\{[\xi:c]\in\Dn:c\neq0\bigr\}
 =
 \bigl\{[w:1]:w\in V\setminus\{0\}\bigr\}.
\end{equation}
The standard affine chart of $\mathbb P(V\oplus\Fq)$ determined by
$c\neq0$ is identified with $V$ through
$[\xi:c]\mapsto c^{-1}\xi$.  The vector $0\in V$ would correspond to
$[0:1]$, which is not a refined direction because its spatial component is zero. Thus, $\Dn^{\mathrm{aff}}$ is identified with the punctured affine space $V\setminus\{0\}$.  More precisely, the map
\[
 V\setminus\{0\}\longrightarrow\Dn^{\mathrm{aff}},
 \qquad
 w\longmapsto[w:1],
\]
is a bijection. 

\begin{lemma}
\label{lem:normalized-horizontal-lines}
Fix $w\in V\setminus\{0\}$.  The assignment
\[
 \rho\in V\longmapsto L_w(\rho),
 \qquad
 L_w(\rho)
 :=
 \bigl\{(\rho+sw,s):s\in\Fq\bigr\},
\]
is a bijection from
\[
 \bigl\{\rho\in V:\sigma(\rho,w)=1\bigr\}
\]
onto the family of horizontal lines with refined direction $[w:1]$.
\end{lemma}

\begin{proof}
Let $L=L_{(z_0,t_0),\xi}$ be a horizontal line with refined direction
$[w:1]$, and put
\[
 c:=\sigma(z_0,\xi).
\]
By \eqref{eq:horizontal-line} and \eqref{eq:refined-direction},
\[
 L
 =
 \bigl\{(z_0+\zeta\xi,t_0+\zeta c):\zeta\in\Fq\bigr\},
 \qquad
 \Dir(L)=[\xi:c].
\]
Since $[\xi:c]=[w:1]$, comparison of the last homogeneous coordinates gives $c\neq0$ and $\xi=cw$.

We now use the central coordinate as the line parameter.  Put
\[
 s:=t_0+\zeta c.
\]
Since $c\neq0$, the variable $s$ ranges over all of $\Fq$, and
$\zeta=c^{-1}(s-t_0)$.  Therefore,
\begin{align*}
 (z_0+\zeta\xi,t_0+\zeta c)
 &=\bigl(z_0+(s-t_0)c^{-1}\xi,s\bigr)\\
 &=\bigl(z_0-t_0w+sw,s\bigr).
\end{align*}
With
\[
 \rho:=z_0-t_0w,
\]
this gives
\begin{equation}\label{eq:normalized-line}
 L=L_w(\rho)
 =
 \bigl\{(\rho+sw,s):s\in\Fq\bigr\}.
\end{equation}
Moreover, since $\sigma$ is alternating,
\begin{align*}
 \sigma(\rho,w)
 &=\sigma(z_0-t_0w,w)\\
 &=\sigma(z_0,w)-t_0\sigma(w,w)\\
 &=c^{-1}\sigma(z_0,\xi)
 =1.
\end{align*}

Conversely, suppose that $\sigma(\rho,w)=1$.  The horizontal line through
$(\rho,0)$ with spatial direction $w$ is
\[
 L_{(\rho,0),w}
 =
 \bigl\{(\rho+sw,s\sigma(\rho,w)):s\in\Fq\bigr\}
 =L_w(\rho),
\]
and its refined direction is $[w:1]$.

Finally, $w$ is determined by the normalized direction $[w:1]$. Once
$w$ is fixed, $L_w(\rho)$ meets $V\times\{0\}$ at the unique point
$(\rho,0)$, so $\rho$ is also uniquely determined.
\end{proof}

Since $\sigma$ is nondegenerate and $w\neq0$, the map
$\rho\mapsto\sigma(\rho,w)$ is a nonzero linear functional on $V$.
Consequently, $\{\rho:\sigma(\rho,w)=1\}$ is an affine hyperplane with
$q^{r-1}$ points. Thus, every direction $[w:1]\in\Dn^{\mathrm{aff}}$ is represented by exactly $q^{r-1}$ horizontal lines. In particular,
\begin{equation}\label{eq:affine-chart-cardinality}
 |\Dn^{\mathrm{aff}}|=q^r-1.
\end{equation}
The advantage of \eqref{eq:normalized-line} is that the variable $s$ used to parametrize the line is exactly the central coordinate of the point $(\rho+sw,s)$. Fix coordinates on $V$, and write
$\mathbf Z=(Z_1,\ldots,Z_r)$ for the corresponding formal variables and
$U$ for the central formal variable. If $P\in\Fq[\mathbf Z,U]$, then its restriction to $L_w(\rho)$ is the univariate polynomial
\[
 s\longmapsto P(\rho+sw,s).
\]
For $E\subset\HH_n(\Fq)$, define the set of normalized vector parameters
\begin{equation}\label{eq:normalized-rich-vectors}
 \widetilde\Omega_\lambda(E)
 :=
 \left\{w\in V\setminus\{0\}:
 \max_{\substack{\rho\in V\\\sigma(\rho,w)=1}}
 |L_w(\rho)\cap E|\geq\lambda
 \right\}.
\end{equation}

Thus, $\widetilde\Omega_\lambda(E)$ is a subset of $V\setminus\{0\}$,
not a subset of the projective direction space. Its elements are the
unique normalized vector representatives of the rich directions in
$\Dn^{\mathrm{aff}}$.
By \cref{lem:normalized-horizontal-lines}, the maximum in
\eqref{eq:normalized-rich-vectors} is taken over all horizontal lines of refined direction $[w:1]$. Hence,
\begin{equation}\label{eq:normalized-rich-cardinality}
 |\widetilde\Omega_\lambda(E)|
 =
 \bigl|\Omega_\lambda(E)\cap\Dn^{\mathrm{aff}}\bigr|.
\end{equation}

\subsection{Horizontal multiplicity}
The purpose of this subsection is to construct a polynomial whose
restriction to every normalized horizontal line introduced above vanishes, with prescribed multiplicity, at the parameters corresponding to points of $E$. More precisely, let $E\subset\HH_n(\Fq)=V\times\Fq$ be nonempty, and let $m\geq1$ be an integer. We shall prove that there is a single nonzero polynomial $P\in\Fq[\mathbf Z,U]$ satisfying
\begin{equation}\label{eq:horizontal-multiplicity-overview-degree}
 \deg P
 \leq
 C_{\mathrm{int}}|E|^{1/(r+1)}m^{r/(r+1)},
\end{equation}
where $C_{\mathrm{int}}=C_{\mathrm{int}}(n)$, with the following property.
For every $w\in V\setminus\{0\}$, every $\rho\in V$ satisfying
$\sigma(\rho,w)=1$, and every $s_0\in\Fq$ for which
\[
 p=(\rho+s_0w,s_0)\in L_w(\rho)\cap E,
\]
the univariate restriction
\[
 R_{\rho,w}(s):=P(\rho+sw,s)
\]
satisfies
\begin{equation}\label{eq:horizontal-multiplicity-overview-root}
 (s-s_0)^m\mid R_{\rho,w}(s).
\end{equation}
Equivalently, $R_{\rho,w}$ has a zero of Hasse multiplicity at least $m$
at every parameter $s_0$ corresponding to a point of $L_w(\rho)\cap E$. Thus, the same polynomial $P$ works simultaneously for all normalized horizontal lines and all their intersection points with $E$.

Let us first explain the relevant space of directions at a point.  If
$p=(z,t)$ and $w\in V\setminus\{0\}$, then the horizontal line through $p$ with spatial direction $w$ is
\[
 s\longmapsto\bigl(z+sw,t+s\sigma(z,w)\bigr).
\]
Its direction vector in $V\times\Fq$ is $(w,\sigma(z,w))$.

For $p=(z,t)\in V\times\Fq$, let
\begin{equation}\label{eq:horizontal-tangent-space}
 \mathcal H_p
 :=
 \bigl\{(\xi,\sigma(z,\xi)):\xi\in V\bigr\}.
\end{equation}
Thus, $\mathcal H_p$ consists precisely of the direction vectors of
horizontal lines through $p$, together with the zero vector.  It is the graph of the linear functional $\xi\mapsto\sigma(z,\xi)$ and is therefore an $r$-dimensional linear subspace of $V\times\Fq$.

\begin{definition}
\label{def:horizontal-multiplicity}
Let $P\in\Fq[\mathbf Z,U]$, and let $m$ be a positive integer.  At
$p\in V\times\Fq$, write the Hasse--Taylor expansion
\begin{equation}\label{eq:Hasse-expansion}
 P(p+h)=\sum_{j\geq0}P_p^{(j)}(h),
\end{equation}
where $P_p^{(j)}$ is homogeneous of degree $j$ in
$h\in V\times\Fq$. We write
\[
 \operatorname{mult}^{\mathrm{hor}}_p(P)\geq m
\]
if, for every $0\leq j<m$, the restriction of $P_p^{(j)}$ to
$\mathcal H_p$ is the zero polynomial.
\end{definition}
We record explicitly how horizontal multiplicity is related to the
Hasse--Taylor expansion of the restriction of $P$ to the horizontal
directions at a point.

Fix $p=(z,t)\in\HH_n(\Fq)$. Introduce an ambient increment
\[
 h=(X,T)\in V\times\Fq.
\]
The Hasse--Taylor expansion of $P$ at $p$ is the formal polynomial
identity
\begin{equation}\label{eq:hasse-taylor-at-p}
 P(p+h)
 =
 \sum_{\alpha\in\mathbb N^{2n+1}}
 \partial^{[\alpha]}P(p)\,h^\alpha
 =
 \sum_{j\geq0}P_{p,j}(h),
\end{equation}
where $\partial^{[\alpha]}P$ denotes the Hasse derivative and
\[
 P_{p,j}(h)
 :=
 \sum_{|\alpha|=j}
 \partial^{[\alpha]}P(p)\,h^\alpha.
\]
Thus, $P_{p,j}$ is a homogeneous polynomial of degree $j$ in the
ambient increment $h$.

The linear space of horizontal directions at $p$ is
\[
 \mathcal H_p
 =
 \bigl\{(v,\sigma(z,v)):v\in V\bigr\}.
\]
It is parametrized by the linear map
\[
 \iota_p:V\longrightarrow\mathcal H_p,
 \qquad
 \iota_p(v):=(v,\sigma(z,v)).
\]
We restrict $P$ to the affine tangent slice determined by the
horizontal directions at $p$ by setting
\[
 \widetilde P_p(v)
 :=
 P\bigl(p+\iota_p(v)\bigr)
 =
 P\bigl(z+v,t+\sigma(z,v)\bigr).
\]
Substituting $h=\iota_p(v)$ into the formal identity
\eqref{eq:hasse-taylor-at-p} gives
\begin{equation}\label{eq:restricted-hasse-taylor}
 \widetilde P_p(v)
 =
 \sum_{j\geq0}
 P_{p,j}\bigl(\iota_p(v)\bigr)
 =
 \sum_{j\geq0}
 P_{p,j}\bigl(v,\sigma(z,v)\bigr)
\end{equation}
as an identity in the polynomial ring $\Fq[V]$.

For each $j$, the polynomial
\[
 v\longmapsto P_{p,j}\bigl(v,\sigma(z,v)\bigr)
\]
is homogeneous of degree $j$. Indeed, since $z$ is fixed, the map
$v\mapsto\iota_p(v)$ is linear, and hence, for every $a\in\Fq$,
\[
 P_{p,j}\bigl(\iota_p(av)\bigr)
 =
 P_{p,j}\bigl(a\iota_p(v)\bigr)
 =
 a^jP_{p,j}\bigl(\iota_p(v)\bigr).
\]
Equivalently, if
\[
 P_{p,j}(X,T)
 =
 \sum_{|\beta|+\ell=j}
 c_{\beta,\ell}X^\beta T^\ell,
\]
then
\[
 P_{p,j}\bigl(v,\sigma(z,v)\bigr)
 =
 \sum_{|\beta|+\ell=j}
 c_{\beta,\ell}v^\beta\sigma(z,v)^\ell,
\]
and every term on the right has total degree
$|\beta|+\ell=j$ in $v$.

It follows from \eqref{eq:restricted-hasse-taylor} that the
degree-$j$ homogeneous part of $\widetilde P_p$ is precisely
\[
 P_{p,j}\big|_{\mathcal H_p}
 :=
 P_{p,j}\circ\iota_p.
\]
Because terms of distinct homogeneous degrees cannot cancel one
another, we obtain
\begin{equation}\label{eq:horizontal-multiplicity-restriction}
 \operatorname{mult}^{\mathrm{hor}}_p(P)\geq m
 \quad\Longleftrightarrow\quad
 P_{p,j}\circ\iota_p=0
 \text{ in }\Fq[V]
 \quad\text{for every }0\leq j<m.
\end{equation}
Equivalently,
\[
 \operatorname{mult}^{\mathrm{hor}}_p(P)\geq m
 \quad\Longleftrightarrow\quad
 \widetilde P_p
 \in (v_1,\ldots,v_{2n})^m.
\]
Here $(v_1,\ldots,v_{2n})$ denotes the ideal of the origin in
$\Fq[v_1,\ldots,v_{2n}]$. Its $m$th power consists precisely of
the polynomials all of whose monomials have total degree at least $m$.
Thus, $\operatorname{mult}^{\mathrm{hor}}_p(P)$ is the ordinary Hasse multiplicity at $v=0$ of the polynomial $\widetilde P_p$. In other words, it measures the order to which $P$ vanishes along all horizontal
directions based at $p$. Here the condition
\[
 P_{p,j}\circ\iota_p=0\quad\text{in }\Fq[V]
\]
means that $P_{p,j}\circ\iota_p$ is the zero formal polynomial, namely
that all of its coefficients vanish. It is stronger than requiring
that the associated polynomial function vanish at every point of
$V(\Fq)$.  This distinction is necessary over finite fields, where a
nonzero formal polynomial may induce the zero function.

\begin{lemma}
\label{lem:multiplicity-condition-count}
For every positive integer $m$ and every fixed $p$, the condition
$\operatorname{mult}^{\mathrm{hor}}_p(P)\geq m$ imposes at most
\begin{equation}\label{eq:multiplicity-condition-count}
 \binom{m+r-1}{r}
\end{equation}
homogeneous linear conditions on the coefficients of $P$.
\end{lemma}

\begin{proof}
Write $p=(z,t)$ and parametrize $\mathcal H_p$ by the injective linear map
\[
 \iota_p:V\longrightarrow\mathcal H_p,
 \qquad
 \iota_p(\xi):=(\xi,\sigma(z,\xi)).
\]
For each $j\geq0$, put
\[
 Q_p^{(j)}(\xi)
 :=
 P_p^{(j)}\bigl(\xi,\sigma(z,\xi)\bigr).
\]
Because $\iota_p$ is linear and $P_p^{(j)}$ is homogeneous of degree $j$, $Q_p^{(j)}$ is a homogeneous polynomial of degree $j$ in the $r$ coordinates of $\xi$. The condition $P_p^{(j)}|_{\mathcal H_p}=0$ says precisely that every coefficient of the formal polynomial $Q_p^{(j)}$ is zero. There are at most $\binom{j+r-1}{r-1}$ such coefficients, and each is a homogeneous linear expression in the coefficients of $P$. Summing over $0\leq j<m$ gives
\[
 \sum_{j=0}^{m-1}\binom{j+r-1}{r-1}
 =\binom{m+r-1}{r}.
\]
The resulting equations need not be independent, which is why the lemma gives an upper bound for the number of conditions.
\end{proof}

For comparison, ordinary multiplicity $m$ in the $r+1$ ambient variables
would impose $\binom{m+r}{r+1}=O_n(m^{r+1})$ conditions at one point.
Horizontal multiplicity imposes only $\binom{m+r-1}{r}=O_n(m^r)$ conditions.

\begin{lemma}
\label{lem:horizontal-interpolation}
There is a constant $C_{\mathrm{int}}=C_{\mathrm{int}}(n)\geq1$ with
the following property. Let $E\subset V\times\Fq$ be nonempty, and let $m\geq1$ be an integer. Then there is a nonzero polynomial
$P\in\Fq[\mathbf Z,U]$ such that
\[
 \operatorname{mult}^{\mathrm{hor}}_p(P)\geq m
 \qquad(p\in E)
\]
and
\begin{equation}\label{eq:horizontal-interpolation-degree}
 \deg P
 \leq
 C_{\mathrm{int}}|E|^{1/(r+1)}m^{r/(r+1)}.
\end{equation}
\end{lemma}

\begin{proof}
Set $N=|E|$.
Let $\operatorname{Poly}_{\leq D}$ be the vector space of polynomials in the
$r+1$ formal variables $(\mathbf Z,U)$ of total degree at most $D$.  Then
\[
 \dim\operatorname{Poly}_{\leq D}
 =
 \binom{D+r+1}{r+1}
 \geq
 \frac{D^{r+1}}{(r+1)!}.
\]
By \cref{lem:multiplicity-condition-count}, horizontal multiplicity at least $m$ at all points of $E$ imposes at most
\[
 N\binom{m+r-1}{r}
 \leq C_{\mathrm{cond}}Nm^r
\]
homogeneous linear conditions, for a constant
$C_{\mathrm{cond}}=C_{\mathrm{cond}}(n)$. Choose
$C_{\mathrm{deg}}=C_{\mathrm{deg}}(n)\geq1$ so large that
\[
 \frac{C_{\mathrm{deg}}^{r+1}}{(r+1)!}>C_{\mathrm{cond}},
\]
and set
\[
 D_0
 :=
 \left\lceil C_{\mathrm{deg}}N^{1/(r+1)}m^{r/(r+1)}\right\rceil.
\]
Then $\dim\operatorname{Poly}_{\leq D_0}>C_{\mathrm{cond}}Nm^r$.
The corresponding homogeneous linear system therefore has a nonzero solution $P\in\operatorname{Poly}_{\leq D_0}$. Since
$N^{1/(r+1)}m^{r/(r+1)}\geq1$, the ceiling can be absorbed into a constant $C_{\mathrm{int}}$ depending only on $n$, which proves
\eqref{eq:horizontal-interpolation-degree}.
\end{proof}

\begin{lemma}
\label{lem:line-multiplicity}
Let $P\in\Fq[\mathbf Z,U]$ and let $m\geq1$ be an integer.  Fix
$w\in V\setminus\{0\}$ and $\rho\in V$ with $\sigma(\rho,w)=1$, and let
$p=(\rho+s_0w,s_0)\in L_w(\rho)$.  If
$\operatorname{mult}^{\mathrm{hor}}_p(P)\geq m$, then
\[
 R_{\rho,w}(s):=P(\rho+sw,s)
\]
has a zero of Hasse multiplicity at least $m$ at $s=s_0$.
\end{lemma}

\begin{proof}

Since $\sigma(\rho,w)=1$ and $\sigma(w,w)=0$,
\[
 \sigma(\rho+s_0w,w)
 =
 \sigma(\rho,w)+s_0\sigma(w,w)
 =1.
\]
Thus, $(w,1)\in\mathcal H_p$.  Using a scalar formal increment $h_0$, we have
\begin{align*}
 R_{\rho,w}(s_0+h_0)
 &=P\bigl(\rho+(s_0+h_0)w,s_0+h_0\bigr)\\
 &=P\bigl(p+h_0(w,1)\bigr)\\
 &=\sum_{j\geq0}P_p^{(j)}\bigl(h_0(w,1)\bigr)\\
 &=\sum_{j\geq0}h_0^jP_p^{(j)}(w,1).
\end{align*}
The coefficient of $h_0^j$ in the Hasse expansion of $R_{\rho,w}$ at $s_0$ is therefore $P_p^{(j)}(w,1)$. It vanishes for every $j<m$ because
$(w,1)\in\mathcal H_p$ and $\operatorname{mult}^{\mathrm{hor}}_p(P)\geq m$. Consequently, the above expansion contains only terms of degree at least $m$ in $h_0$, and hence,
\[
 h_0^m\mid R_{\rho,w}(s_0+h_0)
 \qquad\text{in }\Fq[h_0].
\]
Replacing $h_0$ by $s-s_0$, we obtain
\[
 (s-s_0)^m\mid R_{\rho,w}(s).
\]
Thus, $s_0$ is a zero of $R_{\rho,w}$ of multiplicity at least $m$, as
claimed.
\end{proof}

\begin{lemma}
\label{lem:leading-term-on-line}
Let $P\in\Fq[\mathbf Z,U]$ be nonzero of degree $D$, and let
$P^{\mathrm{top}}$ be its homogeneous part of degree $D$. For
$\rho,w\in V$, define
\[
 R_{\rho,w}(s):=P(\rho+sw,s).
\]
Set
\[
 \Gamma_P(\mathbf Z):=P^{\mathrm{top}}(\mathbf Z,1).
\]
Then the coefficient of $s^D$ in $R_{\rho,w}$ is $\Gamma_P(w)$.
Moreover, $\Gamma_P$ is a nonzero formal polynomial of degree at most $D$. 
\end{lemma}

\begin{proof}
Every homogeneous part of $P$ of degree strictly less than $D$ contributes only powers of $s$ strictly below $D$ after the substitution
$(\mathbf Z,U)=(\rho+sw,s)$. Write
\[
 P^{\mathrm{top}}(\mathbf Z,U)
 =
 \sum_{|\delta|+j=D}c_{\delta,j}\mathbf Z^\delta U^j.
\]
Here $|\delta|$ means the sum of coordinates of $\delta$. 

For the term
\[
 c_{\delta,j}(\rho+sw)^\delta s^j,
\]
the coefficient of $s^{|\delta|+j}=s^D$ is
$c_{\delta,j}w^\delta$. Summing over \((\delta,j)\), the coefficient of \(s^D\) in \(P(\rho+sw,s)\) is
\[
 \sum_{|\delta|+j=D}c_{\delta,j}w^\delta
 =P^{\mathrm{top}}(w,1)
 =\Gamma_P(w).
\]
In particular, this coefficient is independent of $\rho$.

Since \(P^{\mathrm{top}}\) is a nonzero homogeneous polynomial of
degree \(D\), every monomial occurring in \(P^{\mathrm{top}}\) is
uniquely of the form \(\mathbf Z^\delta U^{D-|\delta|}\). Hence,
\[
 P^{\mathrm{top}}(\mathbf Z,U)
 =
 \sum_{\substack{\delta\in\mathbb N^r\\|\delta|\leq D}}
 a_\delta \mathbf Z^\delta U^{D-|\delta|},
\]
where \(a_\delta\in\Fq\) and the coefficients \(a_\delta\) are not all
zero. Setting \(U=1\), we obtain
\[
 \Gamma_P(\mathbf Z)
 =
 P^{\mathrm{top}}(\mathbf Z,1)
 =
 \sum_{\substack{\delta\in\mathbb N^r\\|\delta|\leq D}}
 a_\delta\mathbf Z^\delta.
\]
Since the monomials \(\mathbf Z^\delta\) corresponding to distinct
multi-indices \(\delta\) are distinct, no cancellation can occur in
this formal polynomial. Therefore,
\(\Gamma_P\not\equiv 0\), and \(\deg\Gamma_P\leq D\).
\end{proof}

\subsection{The rich-direction estimate}
In this subsection, we prove the following result.
\begin{proposition}
\label{prop:polynomial-rich-line}
Let $E\subset\HH_n(\Fq)$, and let $\lambda\geq1$.  Then
\begin{equation}\label{eq:polynomial-rich-line}
 |\widetilde\Omega_\lambda(E)|
 \lesssim_n
 q^{r-1}
 \max\left\{
 |E|^{1/(r+1)},\frac{|E|}{\lambda^r}
 \right\}.
\end{equation}
\end{proposition}

The term $q^{r-1}|E|^{1/(r+1)}$ in
\eqref{eq:polynomial-rich-line} prevents a direct deduction of
\eqref{eq:rich-direction-bound} when $|E|<\lambda^{r+1}$.  We avoid its effect by a density-amplification argument.  Starting from a nonempty rich-direction set $\Omega_\lambda(E)$,
\cref{lem:covering-direction-set} provides
$O(q^r/|\Omega_\lambda(E)|)$ affine symplectic images whose rich
directions cover a positive proportion of the affine chart.  By covariance, the union $E_{\mathrm{cov}}$ of the corresponding images of $E$ is $\lambda$-rich in all these directions.  Applying
\cref{prop:polynomial-rich-line} to $E_{\mathrm{cov}}$, the first term gives $|E_{\mathrm{cov}}|\gtrsim_n q^{r+1}\geq q\lambda^r$, while the second gives $|E_{\mathrm{cov}}|\gtrsim_n q\lambda^r$ directly; thus, the additional term is harmless in the positive-density regime. Since $|E_{\mathrm{cov}}|\lesssim q^r|E|/|\Omega_\lambda(E)|$, this yields \eqref{eq:rich-direction-bound}. The general device of passing from a sparse direction family to a dense one by superimposing random symmetry images appears in \cite[Proposition~2.5 and Remark~3.2]{EOT10}; a particularly close set-level random-rotation argument appears in
\cite[Proposition~6]{Salvatore23}. The details are given in
\cref{sec:covering-directions,sec:global-rich-direction}.

To prove this proposition, we first record two elementary facts that will be used to convert vanishing on rich
horizontal lines into a bound for the number of their refined directions. The first statement is clear from the degree of the polynomial, and the second is the Schwartz--Zippel
lemma~\cite{Schwartz80, Zippe79}.

\begin{lemma}
\label{lem:polynomial-zero-bounds}
Let $\mathbb K$ be a field, and let $d\geq1$.
\begin{enumerate}[label=\textup{(\roman*)}]
\item If $f\in\mathbb K[X]$ is nonzero, then the sum of its Hasse
multiplicities at distinct points of $\mathbb K$ is at most $\deg f$.
\item If $H\in\Fq[X_1,\ldots,X_d]$ is nonzero and has total degree at most $D$, then
\[
 |\{x\in\Fq^d:H(x)=0\}|\leq Dq^{d-1}.
\]
\end{enumerate}
\end{lemma}

\begin{proof}[Proof of Proposition \ref{prop:polynomial-rich-line}]
The result is immediate if $E=\varnothing$. It is also
immediate if $\lambda>\min\{q,|E|\}$, since every horizontal line has $q$ points and contains at most $|E|$ points of $E$. We may therefore assume that $|E|\geq1$ and $1\leq\lambda\leq\min\{q,|E|\}$.

Let $C_{\mathrm{int}}$ be the constant in \cref{lem:horizontal-interpolation}. Choose $C_{\mathrm{gap}}=C_{\mathrm{gap}}(n)\geq1$ so large that
\begin{equation}\label{eq:multiplicity-gap-constant}
 C_{\mathrm{int}}C_{\mathrm{gap}}^{-1/(r+1)}<1,
\end{equation}
and put
\[
 m_0
 :=
 \max\left\{1,\frac{|E|}{\lambda^{r+1}}\right\},
 \qquad
 m:=\lceil C_{\mathrm{gap}}m_0\rceil.
\]
By \cref{lem:horizontal-interpolation}, there is a nonzero polynomial $P\in\Fq[\mathbf Z,U]$ such that
\begin{equation}\label{eq:interpolating-multiplicity}
 \operatorname{mult}^{\mathrm{hor}}_p(P)\geq m
 \qquad(p\in E)
\end{equation}
and, writing $D:=\deg P$,
\begin{equation}\label{eq:initial-rich-polynomial-degree}
 D
 \leq
 C_{\mathrm{int}}|E|^{1/(r+1)}m^{r/(r+1)}.
\end{equation}
Since $m\geq1$ and $E\neq\varnothing$, these conditions force $P$ to vanish at every point of $E$, so $D\geq1$.

Since $m\geq C_{\mathrm{gap}}|E|/\lambda^{r+1}$, we have
\begin{align*}
 \frac{D}{m\lambda}
 &\leq
 C_{\mathrm{int}}\frac{|E|^{1/(r+1)}m^{r/(r+1)}}{m\lambda}\\
 &=C_{\mathrm{int}}\left(\frac{|E|}{m\lambda^{r+1}}\right)^{1/(r+1)}\\
 &\leq C_{\mathrm{int}}C_{\mathrm{gap}}^{-1/(r+1)}
 <1.
\end{align*}
Thus
\begin{equation}\label{eq:degree-multiplicity-gap}
 D<m\lambda.
\end{equation}

We also record the resulting degree bound. Since $m_0\geq1$,
\[
 m\leq C_{\mathrm{gap}}m_0+1\leq(C_{\mathrm{gap}}+1)m_0,
\]
and therefore
\[
 D\lesssim_n |E|^{1/(r+1)}m_0^{r/(r+1)}.
\]
If $|E|\leq\lambda^{r+1}$, then $m_0=1$ and
$D\lesssim_n|E|^{1/(r+1)}$. If $|E|>\lambda^{r+1}$, then
$m_0=|E|/\lambda^{r+1}$ and
\[
 |E|^{1/(r+1)}m_0^{r/(r+1)}
 =
 |E|^{1/(r+1)}
 \left(\frac{|E|}{\lambda^{r+1}}\right)^{r/(r+1)}
 =
 \frac{|E|}{\lambda^r}.
\]
Consequently,
\begin{equation}\label{eq:final-rich-polynomial-degree}
 D
 \lesssim_n
 \max\left\{|E|^{1/(r+1)},\frac{|E|}{\lambda^r}\right\}.
\end{equation}

Fix $w\in \widetilde\Omega_\lambda(E)$. By definition, there is $\rho_w\in V$ such that
\[
 \sigma(\rho_w,w)=1,
 \qquad
 |L_w(\rho_w)\cap E|\geq\lambda.
\]
Let
\[
 \mathcal A_w
 :=
 \{s\in\Fq:(\rho_w+sw,s)\in E\}.
\]
The normalized parametrization gives
\[
 |\mathcal A_w|=|L_w(\rho_w)\cap E|\geq\lambda.
\]
Consider
\[
 R_{\rho_w,w}(s):=P(\rho_w+sw,s).
\]
For every $s_0\in\mathcal A_w$, the point $(\rho_w+s_0w,s_0)$ belongs to $E$.
By \cref{lem:line-multiplicity} and
\eqref{eq:interpolating-multiplicity}, the
polynomial $R_{\rho_w,w}$ has a zero of Hasse multiplicity at least $m$ at $s_0$. If $R_{\rho_w,w}$ were nonzero,
\cref{lem:polynomial-zero-bounds} (i) would give
\[
 m|\mathcal A_w|
 \leq
 \sum_{s_0\in\mathcal A_w}\operatorname{mult}_{s_0}(R_{\rho_w,w})
 \leq
 \deg R_{\rho_w,w}
 \leq D.
\]
On the other hand,
\[
 m|\mathcal A_w|\geq m\lambda>D
\]
by \eqref{eq:degree-multiplicity-gap}.  This contradiction proves
that $R_{\rho_w,w}$ is identically zero.

Let $\Gamma_P(\mathbf Z):=P^{\mathrm{top}}(\mathbf Z,1)$, as in
\cref{lem:leading-term-on-line}. That lemma shows that $\Gamma_P$ is a nonzero formal polynomial of degree at most $D$, and that $\Gamma_P(w)$ is the coefficient of $s^D$ in $R_{\rho_w,w}(s)$. Since $R_{\rho_w,w}$ is identically zero, this coefficient vanishes. We have shown
that
\begin{equation}\label{eq:rich-vectors-zero-set}
 \widetilde\Omega_\lambda(E)
 \subset
 \{w\in V:\Gamma_P(w)=0\}.
\end{equation}

\cref{lem:polynomial-zero-bounds} (ii) gives
\[
 |\widetilde\Omega_\lambda(E)|\leq Dq^{r-1}.
\]
Combining this with \eqref{eq:final-rich-polynomial-degree} proves
\eqref{eq:polynomial-rich-line}.
\end{proof}

\begin{corollary}
\label{cor:positive-proportion-directions}
There is a constant $c_n>0$ with the following property.  Let
$E\subset\HH_n(\Fq)$ and $1\leq\lambda\leq q$.  If
\begin{equation}\label{eq:positive-proportion-hypothesis}
 |\widetilde\Omega_\lambda(E)|\geq\frac12q^r,
\end{equation}
then
\begin{equation}\label{eq:positive-proportion-conclusion}
 |E|\geq c_nq\lambda^r.
\end{equation}
\end{corollary}

\begin{proof}
Let $C_{\mathrm{rich}}=C_{\mathrm{rich}}(n)\geq1$ be a constant for which
\cref{prop:polynomial-rich-line} gives
\[
 |\widetilde\Omega_\lambda(E)|
 \leq
 C_{\mathrm{rich}}q^{r-1}
 \max\left\{
 |E|^{1/(r+1)},\frac{|E|}{\lambda^r}
 \right\}.
\]
Combining this with
$|\widetilde\Omega_\lambda(E)|\geq q^r/2$ yields
\[
 \max\left\{
 |E|^{1/(r+1)},\frac{|E|}{\lambda^r}
 \right\}
 \geq
 \frac{q}{2C_{\mathrm{rich}}}.
\]
If
\[
 |E|^{1/(r+1)}
 \geq
 \frac{q}{2C_{\mathrm{rich}}},
\]
then
\[
 |E|
 \geq
 (2C_{\mathrm{rich}})^{-(r+1)}q^{r+1}
 \geq
 (2C_{\mathrm{rich}})^{-(r+1)}q\lambda^r,
\]
where the last inequality follows from $\lambda\leq q$.  Otherwise,
\[
 \frac{|E|}{\lambda^r}
 \geq
 \frac{q}{2C_{\mathrm{rich}}},
\]
and hence
\[
 |E|
 \geq
 (2C_{\mathrm{rich}})^{-1}q\lambda^r.
\]
The conclusion follows by choosing $c_n>0$ smaller than both constants.
\end{proof}

\section{Covering directions by affine symplectic transformations}
\label{sec:covering-directions}

The estimate in the preceding section applies when a positive proportion of the directions in the affine chart $\Dn^{\mathrm{aff}}$ are rich. An arbitrary family of rich directions need not have this property. We now use transformations that preserve horizontal lines to reduce the general case to the affine-chart estimate. First we show that these transformations act transitively on $\Dn$. We then use a probabilistic argument to move copies of any nonempty direction set so that their union covers a positive proportion of $\Dn^{\mathrm{aff}}$.

Let
\[
 \Sp(V,\sigma)
 :=
 \{S\in\operatorname{GL}(V):
   \sigma(Sz,Sw)=\sigma(z,w)\text{ for all }z,w\in V\}.
\]
For $b\in V$, $\tau\in\Fq$, and $S\in\Sp(V,\sigma)$, define
\begin{equation}\label{eq:affine-symplectic-map}
 g_{b,\tau,S}(z,t)
 :=(b,\tau)\cdot(Sz,t)
 =\bigl(b+Sz,\tau+t+\sigma(b,Sz)\bigr).
\end{equation}
The first equality makes the construction transparent: we first apply the symplectic automorphism $(z,t)\mapsto(Sz,t)$ and then translate on the left by $(b,\tau)$. The first map is a group automorphism, and left translation is a bijection; both preserve horizontal lines. These transformations form the finite group
\[
 \mathcal G_n
 :=
 \{g_{b,\tau,S}:b\in V,\ \tau\in\Fq,\ S\in\Sp(V,\sigma)\}.
\]
A direct calculation from \eqref{eq:affine-symplectic-map} gives
\begin{equation}\label{eq:symmetry-composition}
 g_{b,\tau,S}\circ g_{b',\tau',S'}
 =
 g_{b+Sb',\,\tau+\tau'+\sigma(b,Sb'),\,SS'}.
\end{equation}
The identity is $g_{0,0,\operatorname{Id}_V}$.  Using
\eqref{eq:symmetry-composition} and $\sigma(b,b)=0$, we also obtain
\[
 g_{b,\tau,S}^{-1}
 =
 g_{-S^{-1}b,-\tau,S^{-1}}.
\]

\begin{lemma}
\label{lem:symmetry-action}
Every $g=g_{b,\tau,S}\in\mathcal G_n$ maps horizontal lines to
horizontal lines and induces the action
\begin{equation}\label{eq:direction-action}
 g\cdot[w:c]=[Sw:c+\sigma(b,Sw)]
\end{equation}
on $\Dn$. This action is well defined and transitive. Moreover, for
every function $F:\HH_n(\Fq)\to\C$ and every $\vartheta\in\Dn$,
\begin{equation}\label{eq:maximal-covariance}
 \Mrd(F\circ g^{-1})(g\cdot\vartheta)=\Mrd F(\vartheta).
\end{equation}
In particular, for every $E\subset\HH_n(\Fq)$,
\[
 \Mrd\one_{g(E)}(g\cdot\vartheta)=\Mrd\one_E(\vartheta).
\]
\end{lemma}

\begin{proof}
We first compute the image of a horizontal line. Let
$p=(z_0,t_0)$ and
\[
 L_{p,w}
 =
 \{(z_0+sw,t_0+s\sigma(z_0,w)):s\in\Fq\}.
\]
For every $s\in\Fq$,
\begin{align*}
 &g\bigl(z_0+sw,t_0+s\sigma(z_0,w)\bigr)\\
 &\quad=
 \bigl(
 b+Sz_0+sSw,\,
 \tau+t_0+\sigma(b,Sz_0)
 +s[\sigma(z_0,w)+\sigma(b,Sw)]
 \bigr).
\end{align*}
Since $S$ is symplectic,
\begin{align*}
 \sigma(z_0,w)+\sigma(b,Sw)
 &=
 \sigma(Sz_0,Sw)+\sigma(b,Sw)\\
 &=
 \sigma(b+Sz_0,Sw).
\end{align*}
The displayed points therefore form the horizontal line through
$g(p)$ with spatial direction $Sw$. In other words,
\begin{equation}\label{eq:image-horizontal-line}
 g(L_{p,w})=L_{g(p),Sw}.
\end{equation}
If $\Dir(L_{p,w})=[w:c]$, then $c=\sigma(z_0,w)$, so
\eqref{eq:image-horizontal-line} gives
\[
 \Dir\bigl(g(L_{p,w})\bigr)
 =
 [Sw:c+\sigma(b,Sw)].
\]
This proves \eqref{eq:direction-action}.

Replacing $(w,c)$ by $\chi(w,c)$, with $\chi\in\Fq^\times$, multiplies
both coordinates on the right-hand side of \eqref{eq:direction-action} by $\chi$. Thus, the action is well defined on projective classes. Since $Sw\neq0$, its image belongs to $\Dn$.
The central translation parameter $\tau$ does not affect the direction.

We next prove transitivity.  Let $[w:c],[w':c']\in\Dn$. Choose
symplectic bases beginning with $w$ and $w'$, respectively. The linear
map sending the first basis to the second is an element
$S\in\Sp(V,\sigma)$ satisfying $Sw=w'$. By nondegeneracy, the linear functional
\[
 b\longmapsto\sigma(b,w')
\]
is nonzero and hence surjective onto $\Fq$. Choose $b_0\in V$ with
$\sigma(b_0,w')=1$ and set $b=(c'-c)b_0$. Then
\[
 g_{b,0,S}\cdot[w:c]
 =
 [w':c+\sigma(b,w')]
 =
 [w':c'].
\]

Finally, we prove covariance. As $L$ ranges over the horizontal lines
of direction $\vartheta$, \eqref{eq:image-horizontal-line} shows that $g(L)$ ranges bijectively over the horizontal lines of direction
$g\cdot\vartheta$. Hence,
\begin{align*}
 \Mrd(F\circ g^{-1})(g\cdot\vartheta)
 &=
 \max_{\substack{L'\ \mathrm{horizontal}\\
                  \Dir(L')=g\cdot\vartheta}}
 \sum_{p'\in L'}|F(g^{-1}(p'))|\\
 &=
 \max_{\substack{L\ \mathrm{horizontal}\\
                  \Dir(L)=\vartheta}}
 \sum_{p'\in g(L)}|F(g^{-1}(p'))|\\
 &=
 \max_{\substack{L\ \mathrm{horizontal}\\
                  \Dir(L)=\vartheta}}
 \sum_{p\in L}|F(p)|\\
 &=
 \Mrd F(\vartheta).
\end{align*}
This proves \eqref{eq:maximal-covariance}. Taking $F=\one_E$ gives
$F\circ g^{-1}=\one_{g(E)}$ and proves the final assertion.
\end{proof}

For a direction set $\Omega\subset\Dn$, we use
\[
 g\cdot\Omega
 :=\{g\cdot\vartheta:\vartheta\in\Omega\}
\]
for its image under the induced action on $\Dn$.

Recall from \eqref{eq:affine-chart-directions} and
\eqref{eq:affine-chart-cardinality} that
\[
 \Dn^{\mathrm{aff}}
 =
 \{[w:1]:w\in V\setminus\{0\}\},
 \qquad
 |\Dn^{\mathrm{aff}}|=q^r-1.
\]
We first state the covering argument for an arbitrary set of directions;
this form will also be used for weighted functions.

\begin{lemma}
\label{lem:covering-direction-set}
Let $\Omega\subset\Dn$ be nonempty.  There are affine symplectic
transformations $g_1,\ldots,g_J$ such that
\[
 J\lesssim\frac{q^r}{|\Omega|}
\]
and
\[
 \left|
 \Dn^{\mathrm{aff}}\cap\bigcup_{j=1}^Jg_j\cdot\Omega
 \right|
 \geq\frac12q^r.
\]
\end{lemma}

\begin{proof}
By \eqref{eq:cardinalities},
\[
 q^r\leq|\Dn|<2q^r.
\]
Set
\[
 J:=\left\lceil\frac{2|\Dn|}{|\Omega|}\right\rceil
\]
and choose $g_1,\ldots,g_J$ independently and uniformly from
$\mathcal G_n$.

Fix $\vartheta_0\in\Dn$. Transitivity implies that
$g_j^{-1}\cdot\vartheta_0$ is uniformly distributed on $\Dn$, because
all fibers of the orbit map have the same cardinality. Therefore,
\[
 \Pr\bigl(\vartheta_0\in g_j\cdot\Omega\bigr)
 =
 \Pr\bigl(g_j^{-1}\cdot\vartheta_0\in\Omega\bigr)
 =
 \frac{|\Omega|}{|\Dn|}.
\]
By independence,
\begin{align*}
 \Pr\left(
 \vartheta_0\notin\bigcup_{j=1}^Jg_j\cdot\Omega
 \right)
 &=
 \left(1-\frac{|\Omega|}{|\Dn|}\right)^J\\
 &\leq
 \exp\left(-\frac{J|\Omega|}{|\Dn|}\right)
 \leq e^{-2}.
\end{align*}
Summing this probability over
$\vartheta_0\in\Dn^{\mathrm{aff}}$ gives
\[
 \mathbb E\left|
 \Dn^{\mathrm{aff}}\setminus
 \bigcup_{j=1}^Jg_j\cdot\Omega
 \right|
 \leq e^{-2}|\Dn^{\mathrm{aff}}|.
\]
Consequently, some deterministic choice of $g_1,\ldots,g_J$ satisfies
\begin{align*}
 \left|
 \Dn^{\mathrm{aff}}\cap
 \bigcup_{j=1}^Jg_j\cdot\Omega
 \right|
 &\geq(1-e^{-2})|\Dn^{\mathrm{aff}}|\\
 &=(1-e^{-2})(q^r-1)\\
 &\geq\frac12q^r,
\end{align*}
where the last inequality follows from $q^r\geq4$.  Finally, since
$|\Omega|\leq|\Dn|$,
\[
 J
 \leq\frac{2|\Dn|}{|\Omega|}+1
 \leq\frac{3|\Dn|}{|\Omega|}
 <\frac{6q^r}{|\Omega|}.
\]
This completes the proof.
\end{proof}

\section{The global rich-direction estimate}
\label{sec:global-rich-direction}

\begin{proof}[Proof of \cref{thm:rich-direction-estimate}]
If $\Omega_\lambda(E)=\varnothing$, there is nothing to prove.  Since
every horizontal line has $q$ points, the level set is empty when
$\lambda>q$.

Suppose first that $0<\lambda<1$.  For every
$\vartheta\in\Omega_\lambda(E)$, choose a horizontal line
$L_\vartheta$ of direction $\vartheta$ that meets $E$.  By
\cref{lem:line-geometry}, every point belongs to at most
$|\PP^{r-1}(\Fq)|$ selected lines. Hence,
\[
 |\Omega_\lambda(E)|
 \leq\sum_{\vartheta\in\Omega_\lambda(E)}|L_\vartheta\cap E|
 \leq|\PP^{r-1}(\Fq)|\,|E|
 \lesssim q^{r-1}|E|
 \leq q^{r-1}|E|\lambda^{-r}.
\]

It remains to consider $1\leq\lambda\leq q$. Apply
\cref{lem:covering-direction-set} directly to
$\Omega_\lambda(E)$. There are affine symplectic transformations
$g_1,\ldots,g_J$ such that
\[
 J\lesssim\frac{q^r}{|\Omega_\lambda(E)|}
\]
and
\[
 \left|
 \Dn^{\mathrm{aff}}\cap
 \bigcup_{j=1}^Jg_j\cdot\Omega_\lambda(E)
 \right|
 \geq\frac12q^r.
\]
Put
\[
 E_{\mathrm{cov}}:=\bigcup_{j=1}^Jg_j(E).
\]
Suppose that $\vartheta'\in g_j\cdot\Omega_\lambda(E)$. Then
$\vartheta'=g_j\cdot\vartheta$ for some
$\vartheta\in\Omega_\lambda(E)$, and covariance gives
\begin{align*}
 \Mrd\one_{E_{\mathrm{cov}}}(\vartheta')
 &\geq\Mrd\one_{g_j(E)}(\vartheta')\\
 &=\Mrd\one_E(\vartheta)
 \geq\lambda.
\end{align*}
Under the identification
$\Dn^{\mathrm{aff}}=\{[w:1]:w\in V\setminus\{0\}\}$, it follows that
\[
 |\widetilde\Omega_\lambda(E_{\mathrm{cov}})|
 \geq\frac12q^r.
\]
\Cref{prop:polynomial-rich-line} now gives
\[
 q^r
 \lesssim_n q^{r-1}
 \max\left\{
 |E_{\mathrm{cov}}|^{1/(r+1)},
 \frac{|E_{\mathrm{cov}}|}{\lambda^r}
 \right\}.
\]
Consequently, either
\[
 |E_{\mathrm{cov}}|\gtrsim_n q^{r+1}
\]
or
\[
 |E_{\mathrm{cov}}|\gtrsim_n q\lambda^r.
\]
In the first case, $\lambda\leq q$ implies
$q^{r+1}\geq q\lambda^r$. Thus, in both cases
\[
 |E_{\mathrm{cov}}|\gtrsim_n q\lambda^r.
\]
On the other hand, every $g_j$ is a bijection, and hence
\[
 |E_{\mathrm{cov}}|
 \leq\sum_{j=1}^J|g_j(E)|
 =J|E|
 \lesssim\frac{q^r|E|}{|\Omega_\lambda(E)|}.
\]
Combining the last two estimates and rearranging yields
\[
 |\Omega_\lambda(E)|
 \lesssim_n q^{r-1}|E|\lambda^{-r},
\]
which proves \eqref{eq:rich-direction-bound}.

For sharpness, take $E=\{p_0\}$. Then $\Mrd\one_E=1$ on exactly
$|\PP^{r-1}(\Fq)|\approx q^{r-1}$ refined directions.  
\end{proof}

\section{The sharp level-set estimate for arbitrary functions}
\label{sec:function-level-set}

We next derive the corresponding level-set estimate
for arbitrary functions. For a nonnegative integer-valued function
$\nu:\HH_n(\Fq)\to\mathbb Z_{\geq0}$, put
\[
 \mathcal E_r(\nu)
 :=\sum_{p\in\HH_n(\Fq)}\nu(p)^r
\]
and define
\begin{equation}\label{eq:weighted-normalized-rich-vectors}
 \widetilde\Omega_\lambda(\nu)
 :=
 \left\{w\in V\setminus\{0\}:
 \max_{\substack{\rho\in V\\\sigma(\rho,w)=1}}
 \sum_{p\in L_w(\rho)}\nu(p)\geq\lambda
 \right\}.
\end{equation}
Thus, $\widetilde\Omega_\lambda(\one_E)=\widetilde\Omega_\lambda(E)$.

\begin{proposition}
\label{prop:weighted-affine-rich}
Let $\nu:\HH_n(\Fq)\to\mathbb Z_{\geq0}$.  For every
$\lambda\geq1$,
\begin{equation}\label{eq:weighted-affine-rich}
 |\widetilde\Omega_\lambda(\nu)|
 \lesssim_n
 q^{r-1}
 \max\left\{
 \mathcal E_r(\nu)^{1/(r+1)},
 \frac{\mathcal E_r(\nu)}{\lambda^r}
 \right\}.
\end{equation}
\end{proposition}

\begin{proof}
If $\mathcal E_r(\nu)=0$, then $\nu=0$ and
$\widetilde\Omega_\lambda(\nu)=\varnothing$.  We may therefore assume
that $\mathcal E_r(\nu)\geq1$.  Choose
constants $C_{\mathrm{dim}}=C_{\mathrm{dim}}(n)$ and
$C_{\mathrm{mult}}=C_{\mathrm{mult}}(n)$ below, and set
\[
 m
 :=
 \left\lceil
 C_{\mathrm{mult}}\max\left\{1,
 \frac{\mathcal E_r(\nu)}{\lambda^{r+1}}\right\}
 \right\rceil
\]
and
\[
 D_0
 :=
 \left\lceil
 C_{\mathrm{dim}}\mathcal E_r(\nu)^{1/(r+1)}m^{r/(r+1)}
 \right\rceil.
\]
For every $p$ with $\nu(p)>0$, require
\begin{equation}\label{eq:weighted-horizontal-multiplicity}
 \operatorname{mult}^{\mathrm{hor}}_p(P)\geq m\nu(p).
\end{equation}
By \cref{lem:multiplicity-condition-count}, the total number of
homogeneous linear conditions imposed on $P$ is at most
\[
 \sum_{\substack{p\in\HH_n(\Fq)\\\nu(p)>0}}
 \binom{m\nu(p)+r-1}{r}
 \lesssim_n
 m^r\mathcal E_r(\nu).
\]
The vector space of polynomials in $r+1$ variables of degree at most $D_0$ has dimension
\[
 \binom{D_0+r+1}{r+1}
 \gtrsim_n D_0^{r+1}.
\]
After increasing $C_{\mathrm{dim}}$, this dimension is larger than the number of conditions. Hence, there is a nonzero polynomial
$P\in\Fq[\mathbf Z,U]$ satisfying \eqref{eq:weighted-horizontal-multiplicity}.
Writing $D:=\deg P$, we have $D\leq D_0$.
Since $\mathcal E_r(\nu)>0$, these conditions force $P$ to vanish at some point; hence, $D\geq1$.

We now choose $C_{\mathrm{mult}}$ sufficiently large in terms of
$C_{\mathrm{dim}}$ and $n$. Since
\[
 m\lambda^{r+1}\geq C_{\mathrm{mult}}\mathcal E_r(\nu)
 \qquad\text{and}\qquad
 m\lambda\geq C_{\mathrm{mult}},
\]
we have
\[
 \frac{D_0}{m\lambda}
 \leq
 C_{\mathrm{dim}}
 \left(\frac{\mathcal E_r(\nu)}{m\lambda^{r+1}}\right)^{1/(r+1)}
 +\frac1{m\lambda}
 \leq
 C_{\mathrm{dim}}C_{\mathrm{mult}}^{-1/(r+1)}
 +C_{\mathrm{mult}}^{-1}.
\]
Thus, $C_{\mathrm{mult}}$ may be chosen so that
\begin{equation}\label{eq:weighted-degree-less-than-line-multiplicity}
 D\leq D_0<m\lambda.
\end{equation}
The definitions of $m$ and $D_0$ also give
\begin{equation}\label{eq:weighted-degree-bound}
 D
 \lesssim_n
 \max\left\{
 \mathcal E_r(\nu)^{1/(r+1)},
 \frac{\mathcal E_r(\nu)}{\lambda^r}
 \right\}.
\end{equation}
Indeed, if $\mathcal E_r(\nu)\leq\lambda^{r+1}$, then
$m\approx_n 1$. If $\mathcal E_r(\nu)>\lambda^{r+1}$, then
$m\approx_n \mathcal E_r(\nu)\lambda^{-(r+1)}$.

Fix $w\in \widetilde\Omega_\lambda(\nu)$, and choose $\rho_w$ such that
$\sigma(\rho_w,w)=1$ and
\[
 \sum_{p\in L_w(\rho_w)}\nu(p)\geq\lambda.
\]
For
\[
 R_{\rho_w,w}(s):=P(\rho_w+sw,s),
\]
\cref{lem:line-multiplicity} and
\eqref{eq:weighted-horizontal-multiplicity} show that the total
multiplicity of the zeros of $R_{\rho_w,w}$ is at least
\[
 m\sum_{p\in L_w(\rho_w)}\nu(p)
 \geq m\lambda>D.
\]
Since $\deg R_{\rho_w,w}\leq D$,
\cref{lem:polynomial-zero-bounds} implies that $R_{\rho_w,w}$ is identically zero.

Let $P^{\mathrm{top}}$ be the top homogeneous part of $P$.  For formal variables $\mathbf W=(W_1,\ldots,W_r)$, put
\[
 \Gamma_P(\mathbf W):=P^{\mathrm{top}}(\mathbf W,1).
\]
By \cref{lem:leading-term-on-line}, $\Gamma_P$ is a nonzero polynomial of degree at most $D$, and the coefficient of $s^D$ in $R_{\rho_w,w}(s)$ is $\Gamma_P(w)$. Hence, $\Gamma_P(w)=0$ for every
$w\in \widetilde\Omega_\lambda(\nu)$. By \cref{lem:polynomial-zero-bounds},
\[
 |\widetilde\Omega_\lambda(\nu)|\leq Dq^{r-1}.
\]
Combining this with \eqref{eq:weighted-degree-bound} proves the result.
\end{proof}

\begin{proposition}
\label{prop:integer-weighted-level}
Let $\nu:\HH_n(\Fq)\to\mathbb Z_{\geq0}$, and let
$1\leq\lambda\leq q$.  Then
\begin{equation}\label{eq:integer-weighted-level}
 \left|
 \left\{\vartheta\in\Dn:\Mrd \nu(\vartheta)\geq\lambda\right\}
 \right|
 \lesssim_n
 q^{r-1}\lambda^{-r}\mathcal E_r(\nu).
\end{equation}
\end{proposition}

\begin{proof}
Write
\[
 \Omega_\lambda(\nu)
 :=
 \{\vartheta\in\Dn:\Mrd \nu(\vartheta)\geq\lambda\}.
\]
The assertion is immediate when $\Omega_\lambda(\nu)=\varnothing$.

Apply \cref{lem:covering-direction-set} directly to the set
$\Omega_\lambda(\nu)$.
It gives affine symplectic transformations $g_1,\ldots,g_J$, with
\[
 J\lesssim\frac{q^r}{|\Omega_\lambda(\nu)|},
\]
such that
\begin{equation}\label{eq:weighted-covered-directions}
 \left|
 \Dn^{\mathrm{aff}}\cap\bigcup_{j=1}^J(g_j\cdot\Omega_\lambda(\nu))
 \right|
 \geq\frac12q^r.
\end{equation}
Define
\begin{equation}\label{eq:weighted-superposition}
 \nu_{\mathrm{cov}}(p)
 :=
 \left\lceil
 \left(\sum_{j=1}^J\nu(g_j^{-1}(p))^r\right)^{1/r}
 \right\rceil.
\end{equation}
This is a nonnegative integer-valued function. If the expression inside the ceiling is nonzero, then it is at least $1$. Consequently,
\[
 \nu_{\mathrm{cov}}(p)^r
 \leq
 2^r\sum_{j=1}^J\nu(g_j^{-1}(p))^r.
\]
Since every $g_j$ is a bijection,
\begin{equation}\label{eq:weighted-superposition-norm}
 \mathcal E_r(\nu_{\mathrm{cov}})
 \leq
 2^rJ\mathcal E_r(\nu).
\end{equation}
We claim that every direction in the set on the left-hand side of
\eqref{eq:weighted-covered-directions} is $\lambda$-rich for
$\nu_{\mathrm{cov}}$. Indeed, suppose that $\vartheta'=g_j\cdot\vartheta$ with
$\vartheta\in\Omega_\lambda(\nu)$. Choose a
horizontal line $L$ of direction $\vartheta$ such that
\[
 \sum_{p\in L}\nu(p)\geq\lambda.
\]
Then $g_j(L)$ is a horizontal line of direction $\vartheta'$, and
\[
 \sum_{p\in g_j(L)}\nu_{\mathrm{cov}}(p)
 \geq
 \sum_{p\in g_j(L)}\nu(g_j^{-1}(p))
 =
 \sum_{p\in L}\nu(p)
 \geq\lambda.
\]
Under the identification
$\Dn^{\mathrm{aff}}=\{[w:1]: w\in V\setminus\{0\}\}$,
\eqref{eq:weighted-covered-directions} therefore implies
\[
 |\widetilde\Omega_\lambda(\nu_{\mathrm{cov}})|\geq\frac12q^r.
\]
Applying \cref{prop:weighted-affine-rich} to $\nu_{\mathrm{cov}}$ gives
\[
 q
 \lesssim_n
 \max\left\{
 \mathcal E_r(\nu_{\mathrm{cov}})^{1/(r+1)},
 \frac{\mathcal E_r(\nu_{\mathrm{cov}})}{\lambda^r}
 \right\}.
\]
If the second term is bounded below by a fixed multiple of $q$, then
\[
 \mathcal E_r(\nu_{\mathrm{cov}})
 \gtrsim_n q\lambda^r.
\]
If the first term is bounded below by a fixed multiple of $q$, then
\[
 \mathcal E_r(\nu_{\mathrm{cov}})
 \gtrsim_n q^{r+1}
 \geq q\lambda^r,
\]
where the last inequality uses $\lambda\leq q$.  Thus, in either case,
\[
 q\lambda^r
 \lesssim_n
 \mathcal E_r(\nu_{\mathrm{cov}}).
\]
Combining this with \eqref{eq:weighted-superposition-norm} and the bound
for $J$, we obtain
\[
 q\lambda^r
 \lesssim_n
 J\mathcal E_r(\nu)
 \lesssim_n
 \frac{q^r\mathcal E_r(\nu)}{|\Omega_\lambda(\nu)|}.
\]
Rearranging proves \eqref{eq:integer-weighted-level}.
\end{proof}

\begin{proof}[Proof of \cref{thm:function-level-set}]
Since
\[
 \Mrd F=\Mrd|F|
 \qquad\text{and}\qquad
 \norm{F}_{\ell^r(\HH_n(\Fq))}
 =
 \norm{|F|}_{\ell^r(\HH_n(\Fq))},
\]
we may assume that $F\geq0$.  Fix $\lambda>0$, and define
\[
 \nu_\lambda(p)
 :=
 \left\lfloor\frac{2qF(p)}{\lambda}\right\rfloor.
\]
This is a nonnegative integer-valued function.

Suppose that $\Mrd F(\vartheta)\geq\lambda$. Since the collection of horizontal lines of direction $\vartheta$ is finite, there is a horizontal line $L$ of direction $\vartheta$ such that
\[
 \sum_{p\in L}F(p)\geq\lambda.
\]
Every horizontal line has exactly $q$ points. Therefore,
\[
 \sum_{p\in L}\nu_\lambda(p)
 \geq
 \frac{2q}{\lambda}\sum_{p\in L}F(p)-q
 \geq q.
\]
It follows that
\[
 \{\vartheta\in\Dn:\Mrd F(\vartheta)\geq\lambda\}
 \subset
 \{\vartheta\in\Dn:\Mrd \nu_\lambda(\vartheta)\geq q\}.
\]
Applying \cref{prop:integer-weighted-level} with level $q$, we obtain
\begin{align*}
 \left|\{\vartheta\in\Dn:\Mrd F(\vartheta)\geq\lambda\}\right|
 &\lesssim_n
 q^{r-1}q^{-r}\mathcal E_r(\nu_\lambda)\\
 &\leq
 q^{r-1}q^{-r}
 \left(\frac{2q}{\lambda}\right)^r
 \sum_{p\in\HH_n(\Fq)}F(p)^r\\
 &\lesssim_n
 q^{r-1}\lambda^{-r}
 \norm{F}_{\ell^r(\HH_n(\Fq))}^r.
\end{align*}
This proves \eqref{eq:function-level-set}.

For sharpness, let $F$ be the indicator function of one point. Then $\Mrd F=1$ on $|\PP^{r-1}(\Fq)|\approx q^{r-1}$ refined directions.
\end{proof}

\section{The diagonal estimate}
\label{sec:diagonal}

At the reciprocal exponent pair
\begin{equation}\label{eq:diagonal-point}
 (\alpha,\beta)=\left(\frac1r,\frac1r\right),
 \qquad
 \Phi_n\left(\frac1r,\frac1r\right)=\frac{r-1}{r}.
\end{equation}
The point-mass example gives $\widetilde A_n^{\mathrm{rd}}(1/r,1/r)\geq(r-1)/r$, while \cref{thm:function-level-set} gives the corresponding distributional estimate with the same power of $q$.
\begin{theorem}
\label{thm:diagonal-estimate}
For every $F:\HH_n(\Fq)\to\C$,
\begin{equation}\label{eq:diagonal-estimate}
 \norm{\Mrd F}_{\ell^r(\Dn)}
 \lesssim_n
 q^{(r-1)/r}(1+\log q)^{1/r}
 \norm{F}_{\ell^r(\HH_n(\Fq))}.
\end{equation}
Consequently,
\begin{equation}\label{eq:diagonal-value}
 A_n^{\mathrm{rd}}(r,r)=\frac{r-1}{r}.
\end{equation}
The equality in \eqref{eq:diagonal-value} identifies the
infimum in \cref{def:critical-exponent}; it does not assert that the pure-power estimate at $(r,r)$ is attained.
\end{theorem}

\begin{proof}
List the $|\Dn|$ values of $\Mrd F$ in nonincreasing order:
\[
 a_1\geq a_2\geq\cdots\geq a_{|\Dn|}\geq0.
\]
If $a_j>0$, then at least $j$ directions $\vartheta\in\Dn$ satisfy
$\Mrd F(\vartheta)\geq a_j$.  Applying
\eqref{eq:function-level-set} with $\lambda=a_j$ and rearranging gives
\[
 a_j^r
 \lesssim_n
 \frac{q^{r-1}}{j}
 \norm{F}_{\ell^r(\HH_n(\Fq))}^r.
\]
The last inequality is also immediate when $a_j=0$. Summing over $j$
therefore gives
\[
 \norm{\Mrd F}_{\ell^r(\Dn)}^r
 =
 \sum_{j=1}^{|\Dn|}a_j^r
 \lesssim_n
 q^{r-1}
 \left(\sum_{j=1}^{|\Dn|}\frac1j\right)
 \norm{F}_{\ell^r(\HH_n(\Fq))}^r.
\]
Since $|\Dn|=q+\cdots+q^r<2q^r$,
\[
 \sum_{j=1}^{|\Dn|}\frac1j
 \leq 1+\log|\Dn|
 \lesssim_n
 1+\log q.
\]
Taking the $r$th root proves \eqref{eq:diagonal-estimate}.

It remains to identify the critical exponent. For every $\varepsilon>0$,
\[
 (1+\log q)^{1/r}\lesssim_{\varepsilon,n}q^\varepsilon.
\]
Thus, for every $\varepsilon>0$, the exponent $(r-1)/r+\varepsilon$ is admissible in \cref{def:critical-exponent}. Taking the infimum over the admissible exponents gives
\[
 A_n^{\mathrm{rd}}(r,r)
 \leq
 \frac{r-1}{r}.
\]
Conversely, let $F=\one_{\{p_0\}}$ be the indicator function of one point. Then
\[
 \norm{F}_{\ell^r(\HH_n(\Fq))}=1,
\]
while $\Mrd F=1$ on exactly
$|\PP^{r-1}(\Fq)|\approx q^{r-1}$ refined directions. Hence,
\[
 \norm{\Mrd F}_{\ell^r(\Dn)}
 \approx
 q^{(r-1)/r},
\]
which gives the reverse inequality and proves \eqref{eq:diagonal-value}.
\end{proof}

\begin{proposition}
\label{prop:critical-input-line}
For every $1\leq v\leq\infty$ with $v\neq r$ and every
$F:\HH_n(\Fq)\to\C$,
\begin{equation}\label{eq:critical-input-line}
 \norm{\Mrd F}_{\ell^v(\Dn)}
 \lesssim_{n,v}
 \begin{cases}
  q^{r/v-1/r}\norm{F}_{\ell^r(\HH_n(\Fq))},
    &1\leq v<r,\\[2mm]
  q^{(r-1)/r}\norm{F}_{\ell^r(\HH_n(\Fq))},
    &r<v\leq\infty.
 \end{cases}
\end{equation}
In both cases the displayed power is $q^{\Phi_n(1/r,1/v)}$.
\end{proposition}

\begin{proof}
List the values of $\Mrd F$ as in the proof of
\cref{thm:diagonal-estimate}. The level-set estimate gives
\begin{equation}\label{eq:critical-input-rearrangement}
 a_j
 \lesssim_n
 q^{(r-1)/r}j^{-1/r}
 \norm{F}_{\ell^r(\HH_n(\Fq))}.
\end{equation}
If $r<v<\infty$, then
\[
 \sum_{j=1}^{|\Dn|}j^{-v/r}\lesssim_{r,v}1,
\]
and \eqref{eq:critical-input-rearrangement} gives the second bound in
\eqref{eq:critical-input-line}. The case $v=\infty$ follows directly
from the estimate for $a_1$.

If $1\leq v<r$, then
\[
 \sum_{j=1}^{|\Dn|}j^{-v/r}
 \lesssim_{r,v}
 |\Dn|^{1-v/r}.
\]
Since $|\Dn|<2q^r$, summing \eqref{eq:critical-input-rearrangement} and taking the $v$th root gives
\[
 \norm{\Mrd F}_{\ell^v(\Dn)}
 \lesssim_{n,v}
 q^{(r-1)/r}q^{r(1/v-1/r)}
 \norm{F}_{\ell^r(\HH_n(\Fq))}
 =q^{r/v-1/r}
 \norm{F}_{\ell^r(\HH_n(\Fq))}.
\]
Finally, direct substitution into \eqref{eq:four-functionals} and
\eqref{eq:Phi} gives
\[
 \Phi_n\left(\frac1r,\frac1v\right)
 =
 \begin{cases}
  r/v-1/r,&v<r,\\
  (r-1)/r,&v>r.
 \end{cases}
\]
Only $v=r$ produces the harmonic sum and hence the logarithmic factor.
\end{proof}

\section{The complete range of exponents}
\label{sec:mixed-norm-estimates}

We first record the value of the critical exponent at six reciprocal
exponent pairs.

\begin{lemma}
\label{lem:six-values}
\begin{equation}\label{eq:six-values}
\begin{aligned}
 \widetilde A_n^{\mathrm{rd}}(0,0)&=1,
 &
 \widetilde A_n^{\mathrm{rd}}(1,0)&=0,
 &
 \widetilde A_n^{\mathrm{rd}}(0,1)&=r+1,\\
 \widetilde A_n^{\mathrm{rd}}(1,1)&=r-1,
 &
 \widetilde A_n^{\mathrm{rd}}\left(\frac1r,\frac1r\right)
 &=\frac{r-1}{r},
 &
 \widetilde A_n^{\mathrm{rd}}\left(\frac1r,1\right)
 &=r-\frac1r.
\end{aligned}
\end{equation}
The corresponding strong estimates at $(0,0),(1,0),(0,1),(1,1)$ and $(1/r,1)$ hold with the displayed pure powers. At $(1/r,1/r)$, the value shown is the infimum of the admissible powers, and the available strong estimate has the additional factor
$(1+\log q)^{1/r}$.
\end{lemma}

\begin{proof}
For every selector $\mathcal S$, the elementary estimates
\begin{align}
 \norm{T_{\mathcal S}f}_{\ell^\infty(\Dn)}
 &\leq q\norm{f}_{\ell^\infty(\HH_n(\Fq))},
 \label{eq:value-infty-infty}\\
 \norm{T_{\mathcal S}f}_{\ell^\infty(\Dn)}
 &\leq\norm{f}_{\ell^1(\HH_n(\Fq))},
 \label{eq:value-one-infty}\\
 \norm{T_{\mathcal S}f}_{\ell^1(\Dn)}
 &\leq q|\Dn|\norm{f}_{\ell^\infty(\HH_n(\Fq))}
 \lesssim q^{r+1}\norm{f}_{\ell^\infty(\HH_n(\Fq))},
 \label{eq:value-infty-one}\\
 \norm{T_{\mathcal S}f}_{\ell^1(\Dn)}
 &\leq |\PP^{r-1}(\Fq)|\norm{f}_{\ell^1(\HH_n(\Fq))}
 \lesssim q^{r-1}\norm{f}_{\ell^1(\HH_n(\Fq))}
 \label{eq:value-one-one}
\end{align}
give the asserted upper bounds at the four corners. The lower bounds in \cref{prop:lower-bounds} show that they are sharp. The value at $(1/r,1/r)$ is \cref{thm:diagonal-estimate}.

At $(1/r,1)$, \cref{prop:critical-input-line} with $v=1$
gives the pure-power upper bound
\[
 \norm{\Mrd F}_{\ell^1(\Dn)}
 \lesssim_n
 q^{r-1/r}\norm{F}_{\ell^r(\HH_n(\Fq))}.
\]
Since
\[
 \Phi_n\left(\frac1r,1\right)
 =\phi_3\left(\frac1r,1\right)
 =\phi_4\left(\frac1r,1\right)
 =r-\frac1r.
\]
\cref{prop:lower-bounds} yields
$\widetilde A_n^{\mathrm{rd}}\left(\frac1r,1\right)\geq r-\frac1r$, proving equality.
\end{proof}

The four affine functionals satisfy
\begin{equation}\label{eq:dominance-identities}
\begin{aligned}
 \phi_3-\phi_1&=\beta-\alpha,
 &\qquad
 \phi_4-\phi_2&=r(\beta-\alpha),\\
 \phi_1-\phi_2&=\alpha+(r-1)\beta-1,
 &
 \phi_4-\phi_3&=1-r\alpha.
\end{aligned}
\end{equation}
Define the four cells
\begin{equation}\label{eq:four-cells}
\begin{aligned}
 \mathcal C_2
 &=\conv\left\{(0,0),(1,0),\left(\frac1r,\frac1r\right)\right\},\\
 \mathcal C_1
 &=\conv\left\{(1,0),\left(\frac1r,\frac1r\right),(1,1)\right\},\\
 \mathcal C_4
 &=\conv\left\{(0,0),(0,1),\left(\frac1r,1\right),
                    \left(\frac1r,\frac1r\right)\right\},\\
 \mathcal C_3
 &=\conv\left\{\left(\frac1r,\frac1r\right),
                    \left(\frac1r,1\right),(1,1)\right\}.
\end{aligned}
\end{equation}

\begin{lemma}
\label{lem:four-cells}
The cells in \eqref{eq:four-cells} cover $[0,1]^2$, and
\begin{equation}\label{eq:Phi-on-cells}
 \Phi_n=
 \begin{cases}
  \phi_2,&\text{on }\mathcal C_2,\\
  \phi_1,&\text{on }\mathcal C_1,\\
  \phi_4,&\text{on }\mathcal C_4,\\
  \phi_3,&\text{on }\mathcal C_3.
 \end{cases}
\end{equation}
\end{lemma}

\begin{proof}
Below the diagonal $\beta=\alpha$, the first two identities in
\eqref{eq:dominance-identities} imply
\[
 \Phi_n=\max\{\phi_1,\phi_2\}.
\]
The line $\alpha+(r-1)\beta=1$ meets the boundary of the triangle
$\{0\leq\beta\leq\alpha\leq1\}$ at $(1,0)$ and $(1/r,1/r)$.  By the third identity in
\eqref{eq:dominance-identities}, it separates this triangle into
$\mathcal C_2$, where $\phi_2\geq \phi_1$, and $\mathcal C_1$, where
$\phi_1\geq \phi_2$.

Above the diagonal, the same identities give
\[
 \Phi_n=\max\{\phi_3,\phi_4\}.
\]
The line $\alpha=1/r$ meets the boundary of
$\{0\leq\alpha\leq\beta\leq1\}$ at $(1/r,1/r)$ and $(1/r,1)$.  By the last identity in
\eqref{eq:dominance-identities}, it separates this triangle into
$\mathcal C_4$, where $\phi_4\geq \phi_3$, and $\mathcal C_3$, where $\phi_3\geq \phi_4$.  This proves both the covering and
\eqref{eq:Phi-on-cells}.
\end{proof}

\begin{proof}[Proof of \cref{thm:main}]
The lower bound
\[
 \widetilde A_n^{\mathrm{rd}}(\alpha,\beta)\geq\Phi_n(\alpha,\beta)
\]
is \cref{prop:lower-bounds}.

Suppose first that $(\alpha,\beta)\in\mathcal C_2$. Express
$(\alpha,\beta)$ as a convex combination of $(0,0),(1,0)$, and $(1/r,1/r)$. By \cref{lem:convexity,lem:six-values}, the corresponding convex combination of the values of $\widetilde A_n^{\mathrm{rd}}$ is an upper bound for $\widetilde A_n^{\mathrm{rd}}(\alpha,\beta)$. The affine functional $\phi_2$ agrees with $\Phi_n$ at all three vertices and throughout $\mathcal C_2$. It follows that
\[
 \widetilde A_n^{\mathrm{rd}}(\alpha,\beta)
 \leq \phi_2(\alpha,\beta)=\Phi_n(\alpha,\beta).
\]
The same argument applies to $\mathcal C_1$, using the vertices
$(1,0)$, $(1/r,1/r)$, and $(1,1)$ together with $\phi_1$.  For
$\mathcal C_4$, use a finite convex combination of $(0,0)$, $(0,1)$, $(1/r,1)$, and $(1/r,1/r)$ together with $\phi_4$.  Finally, for $\mathcal C_3$, use $(1/r,1/r)$, $(1/r,1)$, and $(1,1)$ together with $\phi_3$. This proves \eqref{eq:main-phase} and \eqref{eq:main-uv}.

The diagonal estimate \eqref{eq:main-quantitative-diagonal} is
\cref{thm:diagonal-estimate}. It remains to prove the pure-power estimate \eqref{eq:main-quantitative} away from the critical point.  Put
\[
 \begin{gathered}
 P=\left(\frac1r,\frac1r\right),\quad
 O=(0,0),\quad A=(1,0),\\
 B=(1,1),\quad C=(0,1),\quad
 R=\left(\frac1r,1\right).
 \end{gathered}
\]
For every selector $\mathcal S$, \eqref{eq:selector-dominated} and
\cref{thm:function-level-set} give the uniform distributional estimate
\begin{equation}\label{eq:selector-critical-level-set}
 \bigl|\{\vartheta\in\Dn:
       |T_{\mathcal S}f(\vartheta)|\geq\lambda\}\bigr|
 \lesssim_n
 q^{r-1}\lambda^{-r}
 \norm{f}_{\ell^r(\HH_n(\Fq))}^r.
\end{equation}
The elementary estimates \eqref{eq:value-infty-infty}, \eqref{eq:value-one-infty}, and \eqref{eq:value-one-one} are pure-power strong estimates at $O$, $A$, and $B$, respectively.  Applying the Marcinkiewicz interpolation principle \eqref{eq:marcinkiewicz-interpolation} between \eqref{eq:selector-critical-level-set} and these three estimates gives pure-power strong estimates at every point, other than $P$, on the segments $PO$, $PA$, and $PB$. Along $PO$ and $PB$ the interpolated input and output exponents are equal. Along $PA$ they satisfy $u<v$, so the Lorentz-space embedding used in \eqref{eq:marcinkiewicz-interpolation} applies. Because
$\Phi_n$ is affine on each adjacent cell, the interpolated power of $q$ is exactly $\Phi_n$ on all three segments.

By \cref{prop:critical-input-line}, pure-power strong estimates also hold at every point of the segment $PR$ other than $P$.  We now fill the cells using the Riesz--Thorin principle for the selector operators.  For a triangle $\conv\{P,V_1,V_2\}$ and $0<t<1$, set
\[
 V_{i,t}:=(1-t)P+tV_i\qquad(i=1,2).
\]
The truncated quadrilaterals
\[
 \conv\{V_1,V_2,V_{1,t},V_{2,t}\},\qquad 0<t<1,
\]
increase to $\conv\{P,V_1,V_2\}\setminus\{P\}$ as $t\downarrow0$.
Apply this observation to
\[
 \mathcal C_2=\conv\{P,O,A\},\qquad
 \mathcal C_1=\conv\{P,A,B\},\qquad
 \mathcal C_3=\conv\{P,R,B\}.
\]
Every vertex of each truncated quadrilateral has a pure-power strong selector estimate. Since the relevant $\phi_j$ is affine on the cell, finite Riesz--Thorin interpolation gives the power $q^{\Phi_n(\alpha,\beta)}$ at every point of these three cells other than $P$.

Finally, split
\[
 \mathcal C_4
 =\conv\{P,O,R\}\cup\conv\{O,C,R\}.
\]
The first triangle is handled by the same truncation argument using the segments $PO$ and $PR$, while the second has the three pure-power vertices $O,C,R$. Thus, every point of $\mathcal C_4\setminus\{P\}$ also satisfies the required uniform selector estimate. Apply the resulting estimate to $|F|$ and choose a maximizing selector as in \eqref{eq:max-selector}. This proves \eqref{eq:main-quantitative} at every
$(\alpha,\beta)\neq P$.
\end{proof}

\begin{remark}[The logarithmic factor at the critical diagonal]
\label{rem:exact-endpoint}

\setlength{\emergencystretch}{2em}
\setlength{\parskip}{0pt}
The strong mixed-norm estimates have the optimal pure power of $q$ at every pair except possibly $(u,v)=(2n,2n)$. At that point we have
\[
 \norm{\Mrd F}_{\ell^{2n}(\Dn)}
 \lesssim_n
 q^{(2n-1)/(2n)}(1+\log q)^{1/(2n)}
 \norm{F}_{\ell^{2n}(\HH_n(\Fq))}.
\]
The level-set estimate itself has no logarithmic loss.  The factor above arises only from the harmonic sum over the ordered values of the maximal function, as seen in the proof of \cref{thm:diagonal-estimate}; every noncritical sum is handled without it in
\cref{prop:critical-input-line}. Thus, $(2n-1)/(2n)$ is presently known to be the infimum of the admissible diagonal powers, but it is not known to be a minimum. Whether the factor can be removed remains open for $n\geq2$.
When $n=1$, the Fourier argument of~\cite{PPTX26} gives the corresponding logarithm-free diagonal estimate, so the critical power is attained there.
\end{remark}

\raggedbottom

\end{document}